\documentclass[11pt, twoside]{article}

\usepackage{amsmath}
\usepackage{amsthm}
\usepackage{amssymb}
\usepackage{mathrsfs}  
\usepackage[all]{xy}
\usepackage{multirow}
\usepackage{caption}
\usepackage{float}
\usepackage{multicol}
\usepackage{soul}
\usepackage{xcolor}
\usepackage{enumerate}

\usepackage{setspace}
\usepackage{geometry}
\newcommand{\bI}{\textbf{bI}}

\newcommand{\nbI}{\textbf{nbI}}
\newcommand{\nbIciw}{\textbf{nbIciw}}
\newcommand{\nbIci}{\textbf{nbIci}}
\newcommand{\nbIcl}{\textbf{nbIcl}}

\newcommand{\mbC}{\textbf{mbC}}

\newcommand{\mbCci}{\textbf{mbCci}}
\newcommand{\mbCcl}{\textbf{mbCcl}}
\newcommand{\mbCciw}{\textbf{mbCciw}}

\newtheorem{teo}{Theorem}
\newtheorem{mydef}{Definition}
\newtheorem{prop}{Proposition}
\newtheorem{lema}{Lemma}

\newtheorem{cor}{Corollary}

\newtheorem{rem}{Remark}

\newcommand{\lfis}{{\bf LFI}s}
\newcommand{\lfi}{{\bf LFI}}
\newcommand{\cplp}{\text{\bf CPL$^+$}}
\newcommand{\cpl}{\text{\bf CPL}}
\newcommand{\cons}{\ensuremath{{\circ}}}
\newcommand{\mbc}{\textbf{mbC}}
\newcommand{\inc}{\ensuremath{{\uparrow}}}
\newcommand{\MP}{\text{\bf MP}}
\newcommand{\sneg}{\ensuremath{{\sim}}}
\newcommand{\A}{\ensuremath{\mathcal{A}}}

\newcommand{\M}{\ensuremath{\mathcal{M}}}
\newcommand{\F}{\ensuremath{\mathcal{F}}}

\usepackage{authblk}

\title{\textbf{From inconsistency to incompatibility}}
\author{Coniglio, Marcelo E.\thanks{coniglio@unicamp.br} }

\author{Toledo, Guilherme V.\thanks{guivtoledo@gmail.com}}

\affil{Institute of Philosophy and the Humanities - IFCH and\\
Centre for Logic, Epistemology and The History of Science - CLE\\
University of Campinas - Unicamp\\
Campinas, SP, Brazil}

\usepackage{fancyhdr}
\providecommand{\keywords}[1]{\textbf{\textit{Keywords:}} #1}

\begin{document}

\setcounter{page}{1}     

\maketitle

\begin{abstract}
The aim of  this article is to generalize logics of formal inconsistency (\lfis) to systems dealing with the concept of incompatibility, expressed by means of a binary connective. The basic idea is that having two incompatible formulas to hold trivializes a deduction, and as a special case, a formula becomes consistent (in the sense of \lfis) when it is incompatible with its own negation. We show how this notion extends that of consistency in a non-trivial way, presenting conservative translations for many simple \lfis\ into some of the most basic logics of incompatibility, what evidences in a precise way how the notion of incompatibility generalizes that of consistency.
We provide semantics for the new logics, as well as decision procedures, based on restricted non-deterministic matrices. The  use of non-deterministic semantics with restrictions is justified by the fact that, as proved here, these systems are not algebraizable according to Blok-Pigozzi nor are they characterizable by finite Nmatrices. Finally, we briefly compare our logics to other systems focused on treating incompatibility, specially those pioneered by Brandom and further developed by Peregrin. 
\end{abstract}


\section*{Introduction}

Among paraconsistent logics, those of formal inconsistency ($\textbf{LFI}$, \cite{Taxonomy,Handbook, ParLog}) play a prominent role. Their defining property is the mediation of \textit{ex falso quodlibet} by a consistency connective, meaning that $\alpha, \neg\alpha\vdash\beta$ may not longer be true (that is, $\neg$ is a paraconsistent negation), but $\circ\alpha, \alpha, \neg\alpha\vdash\beta$ always holds. In other words, $\alpha$ and $\neg\alpha$, classically, trivialize an argument while, in a given \lfi, one must also assume $\circ\alpha$ in order to trivialize. The formula $\cons\alpha$ intuitively  states that $\alpha$ is `consistent' or `robust' or `classically-behaved' w.r.t. the explosion law of negation, and therefore it satisfies such law (in a local way).

A natural generalization, which we address here, deals with incompatibility. Classically, incompatibility is an important concept, having even its own connective in classical propositional logic ($\textbf{CPL}$), the Sheffer's stroke $\Uparrow$. Inded, from $\alpha{\Uparrow} \beta$, which is equivalent to $\sneg(\alpha \land\beta)$ (with $\sneg$ being the classical negation),  together with $\alpha$ and $\beta$, everything follows (or, in other words, the set $\{\alpha{\Uparrow} \beta,\alpha,\beta\}$ is unsatisfiable in \cpl).  In more contemporary studies (\cite{Brandom, Peregrin1, Peregrin2}), incompatibility has been contemplated as an alternative foundation for (classical) logic, to replace logical deduction, and as one of the cornerstones of epistemology itself. And, although very relevant, these studies never seem to consider the interplay between incompatibility and paraconsistent negations, interplay we show to be most fruitful.

In fact, consider once again the case of \lfis: the controlled explosion law $\circ\alpha, \alpha, \neg\alpha\vdash\gamma$ is made possible by $\circ\alpha$, which asserts the `classicality' or `robustness' of $\alpha$. But that robustness means, essentially, that $\alpha$ and $\neg\alpha$ can not hold together. Let $\alpha\inc\beta$ mean, in a broader interpretation of the Sheffer's stroke $\Uparrow$, simply that $\alpha$ and $\beta$ are logically incompatible, by which we mean that $\alpha\inc\beta, \alpha, \beta\vdash\gamma$. Notice that we make no mention of negation in this definition, although the spirit of incompatibility remains: whenever $\alpha$ and $\beta$ are incompatible, $\alpha$ and $\beta$ together trivialize any argument.
 
 The shape of our axiom is not arbitrary, as it takes inspiration from \lfis. Even more, if one understands the consistency (or classicality) of $\alpha$ as the incompatibility between $\alpha$ and $\neg\alpha$, the more general $\alpha\inc\beta, \alpha, \beta\vdash\gamma$ reduces back to $\cons\alpha, \alpha, \beta\vdash\gamma$  for $\circ\alpha=\alpha\inc\neg\alpha$. Is in this specific sense that we postulate that the logical notion of incompatibility expressed by $\inc$ strictly generalizes the notion of consistency (or classicality) expressed by $\circ$. By analogy with \lfis, a logic with an incompatibily operator \inc\ (primitive or not) will be called a {\em logic of (formal) incompatibility}.

This article is organized as follows: In Section~\ref{first} we present a very short introduction to \lfis\ and some of their most relevant systems. In Section~\ref{second} we introduce the incompatibility connective \inc, as well as the simplest logic of formal incompatibility  based on \cplp, the system \bI. We characterize it semantically, first by use of bivaluations and then by restricted Nmatrices (or RNmatrices), a semantical framework we introduce in~\cite{CostaRNmatrix}.
Based on this, we offer two decision methods for \bI.
In Section~\ref{third} we add a paraconsistent negation to the logics of incompatibility. The first of these systems, \nbI, can be seen as an expansion by the connective \inc\ of \mbc, the basic \lfi\ studied in~\cite{Handbook} and~\cite{ParLog}, by taking $\cons\alpha$ as an abbreviation for $\alpha\inc\neg\alpha$. Then, \nbI\ is extended to $\textbf{nbIciw}$, $\textbf{nbIci}$ and $\textbf{nbIcl}$, whose incompatibilities (expressed by \inc) have power similar to consistency (expressed by $\cons$) in, respectively, $\textbf{mbCciw}$, $\textbf{mbCci}$ and $\textbf{mbCcl}$. To all these systems we offer semantics of bivaluations and RNmatrices, as well as decision methods. In Section~\ref{fourth} several uncharacterizability results for logics of formal  incompatibility are obtained. In particular, it is shown that none of the systems presented here is algebraizable in the sense of Blok and Pigozzi. Moreover, neither \bI\ nor \nbI\ can be characterized by a single finite Nmatrix (recalling that \mbc\ can be characterized by a 5-valued and even by a 3-valued Nmatrix, as shown by Avron in~\cite[Theorem~3.6]{avr:05}). These results justify the use of RNmatrices to deal with the systems presented here. In Section~\ref{fifth} the relation between \lfis\ and logics of formal incompatibility is analyzed by means of (conservative) translations (as defined in~\cite{Itala}). In a sense, this shows that incompatibility indeed strictly generalizes inconsistency. Finally, in Section~\ref{sixth}, a brief comparison between our own systems and those of Brandom is given, stressing their differences. We end the paper by discussing the results obtained here, as well as some possibilities of future research.


\section{The Paraconsistent Logic $\textbf{mbC}$ and Some of  its Extensions}\label{first}

A Tarskian logic $\mathscr{L}$ with a consequence relation $\vdash_{\mathscr{L}}$ is said to be paraconsistent when it possesses a unary connective $\neg$ (primitive or defined), that we shall refer to as a paraconsistent negation, such that there exist formulas $\alpha$ and $\beta$ of $\mathscr{L}$ satisfying $\alpha, \neg\alpha\not\vdash_{\mathscr{L}}\beta$.\footnote{We can be more precise and say that $\mathscr{L}$ is $\neg$-paraconsistent. This makes sense when $\mathscr{L}$ has more than one negation as, for instance, in the case of the logics of formal inconsistency described in this section.}

The paraconsistent logic we will analyze in this paper belong to the class of logics known as {\em logics of formal inconsistency} (\lfis), introduced in~\cite{Taxonomy} (see also~\cite{Handbook,ParLog}). The basic strategy of \lfis\ of controlling the explosion law locally, by means of a `consistency' connective $\circ$, was introduced by Newton da Costa in his landmark treatise~\cite{daCosta}. For instance, in his stronger system $C_1$ the `classicality' (or `well-behavior') of a sentence $\alpha$ is expressed by the formula $\cons\alpha=\neg(\alpha \land \neg\alpha)$. The novelty of the \lfis\ is that such operator \cons\ can be a primitive one, which allows to explore different degrees of paraconsistency.

In addition to having a primitive consistency operator, the logic \mbc, the basic \lfi\ studied in~\cite{Handbook,ParLog},  contains only the positive fragment of classical propositional logic (\cpl) and excluded middle, together with the controlled explosion law mentioned in the Introduction. In this way, \mbc\ is often regarded as the simplest logic of formal inconsistency. It is defined by means of a Hilbert calculus over the signature we will denote by $\Sigma_{\textbf{LFI}}$, given by  $\Sigma_{\textbf{LFI}}=\{\vee, \wedge, \rightarrow,\neg, \circ\}$. The Hilbert calculus for \mbc\ has as axiom schemata

\begin{enumerate}[wide=0pt, leftmargin=*]
\item[\textbf{Ax 1}] \ $\alpha\rightarrow(\beta\rightarrow\alpha)$;
\item[\textbf{Ax 2}] \ $\big(\alpha\rightarrow (\beta\rightarrow \gamma)\big)\rightarrow\big((\alpha\rightarrow\beta)\rightarrow(\alpha\rightarrow\gamma)\big)$;
\item[\textbf{Ax 3}] \ $\alpha\rightarrow\big(\beta\rightarrow(\alpha\wedge\beta)\big)$;
\item[\textbf{Ax 4}] \ $(\alpha\wedge\beta)\rightarrow \alpha$;
\item[\textbf{Ax 5}] \ $(\alpha\wedge\beta)\rightarrow \beta$;
\item[\textbf{Ax 6}] \ $\alpha\rightarrow(\alpha\vee\beta)$;
\item[\textbf{Ax 7}] \ $\beta\rightarrow(\alpha\vee\beta)$;
\item[\textbf{Ax 8}] \ $(\alpha\rightarrow\gamma)\rightarrow\Big((\beta\rightarrow\gamma)\rightarrow \big((\alpha\vee\beta)\rightarrow\gamma\big)\Big)$;
\item[\textbf{Ax 9}] \ $(\alpha\rightarrow \beta)\vee\alpha$;
\item[\textbf{Ax 10}] \ $\alpha\vee\neg \alpha$,
\end{enumerate}

\noindent plus the controlled (or gentle) explosion law
\[\tag{\textbf{bc1}}\circ\alpha\rightarrow(\alpha\rightarrow(\neg \alpha\rightarrow \beta)),\]
and {\em Modus Ponens} (\MP) as the unique inference rule. Other logics of formal inconsistency we will consider here are $\mbCciw$, $\mbCci$ and $\mbCcl$, obtained from the Hilbert system for $\mbC$ by adding, respectively, the axiom schemata

\[\tag{$\textbf{ciw}$}\circ\alpha\vee(\alpha\wedge\neg \alpha)\]
\[\tag{$\textbf{ci}$}\neg {\circ}\alpha\rightarrow(\alpha\wedge\neg \alpha)\]
\[\tag{$\textbf{cl}$}\neg(\alpha\wedge\neg \alpha)\rightarrow\circ\alpha\]


\section{Logics of Incompatibility}\label{second}

In logics with an explosive negation \sneg\ such as, for instance, classical or intuitionistic logic, a formula and its negation are not compatible, in the sense that having both $\alpha$ and ${\sim} \alpha$ to be true (or as hypothesis in a derivation, or as elements of a theory) trivializes the given argument or theory. When dealing with logics of formal inconsistency, this is no longer true: we can have $\alpha$ and the {\em paraconsistent} negation $\neg\alpha$ of $\alpha$ together without trivializing the theory or the argument, unless the additional hypothesis $\circ \alpha$ (meaning that $\alpha$ is `consistent' or `classical') is also present.

To formalize such a notion of incompatibility, we will consider a binary connective that, when connecting formulas $\alpha$ and $\beta$, will stand intuitively for `$\alpha$ is incompatible with $\beta$'. When choosing a symbol for this connective, a natural choice is \inc, in analogy to the Sheffer's stroke $\Uparrow$. The basic axiom we will expect a system for incompatibility to satisfy will be
\[(\alpha\inc\beta)\rightarrow(\alpha\rightarrow(\beta\rightarrow\gamma)),\]
for any formula $\gamma$, or more generally, if we do not have a deduction meta-theorem, $\alpha\inc\beta, \alpha, \beta\vdash_{\mathscr{L}}\gamma$. Intuitively, that means that having $\alpha$ and $\beta$ to be true while having $\alpha$ and $\beta$ to be incompatible implies any formula can be derived: the logic becomes trivial in such circumstances.

Referring back to \lfis, one sees that `consistency' or `classicality' may be characterized as an special case of incompatibility: $\alpha$ is consistent, expressed by $\cons\alpha$, if and only if, $\alpha$ is incompatible with $\neg\alpha$. From now on we will consider the signature $\Sigma_{\bI}=\{\vee, \wedge, \rightarrow, \inc\}$.


\subsection{The Logic $\bI$}

Our quintessential logic of incompatibility, which we shall denote by $\bI$, has only Modus Ponens as  inference rule and consists of the axiom schemata of the positive fragment of classical propositional logic, that is, axiom schemata $1$ trough $9$ of $\textbf{mbC}$, plus
\[\tag{\textbf{Ip}}(\alpha\inc\beta)\rightarrow(\alpha\rightarrow(\beta\rightarrow \gamma))\]
\[\tag{\textbf{Comm}}(\alpha\inc\beta)\rightarrow(\beta\inc\alpha)\]

It is easy to see that $\bI$ has a bottom formula, and so a classical negation  can be defined as usual. Indeed, for any two formulas $\alpha$ and $\beta$ in the language of $\bI$, let $\bot_{\alpha\beta}=(\alpha\wedge\beta)\wedge(\alpha\inc\beta)$.

\begin{lema} \label{L1}
\begin{enumerate}
\item If $\Gamma\vdash_{\bI}\alpha\rightarrow\beta$ and $\Gamma\vdash_{\bI}\beta\rightarrow\gamma$, then $\Gamma\vdash_{\bI}\alpha\rightarrow\gamma$.
\item $\Gamma, \alpha\vdash_{\bI}\beta$ if and only if $\Gamma\vdash_{\bI}\alpha\rightarrow\beta$ (deduction meta-theorem).
\item If $\Gamma, \alpha\vdash_{\bI}\varphi$ and $\Gamma, \beta\vdash_{\bI}\varphi$, then $\Gamma, \alpha\vee\beta\vdash_{\bI}\varphi$ (proof-by-cases).
\end{enumerate}
\end{lema}
\begin{proof}It is straightforward, from the fact that \bI\ contains \cplp, and \MP\ is the unique inference rule.\qedhere
\end{proof}

\begin{prop}
For any formulas $\alpha$, $\beta$ and $\varphi$ it holds: $\bot_{\alpha\beta}\vdash_{\bI}\varphi$.
\end{prop}

\begin{proof}
From $\textbf{Ip}$ and Lemma~\ref{L1}, $\alpha, \beta, \alpha\inc\beta\vdash_{\bI}\varphi$, and by axiom schemata $\textbf{Ax  4}$ and $\textbf{Ax  5}$, we obtain the desired result.\qedhere
\end{proof}

In particular, we find that all such bottom elements are equivalent to each other, so we may define ${\sim}\alpha$ as the formula $\alpha\rightarrow\bot_{\alpha\alpha}$, which is easily seem to have the same properties as the negation of $\textbf{CPL}$ does, namely: $\alpha, \sneg\alpha\vdash_{\bI}\beta$ for every $\alpha$ and $\beta$; $\vdash_{\bI}\alpha \vee \sneg\alpha$; and $\sneg\alpha \vdash_{\bI}\sneg\beta$ whenever  $\beta \vdash_{\bI}\alpha$. By defining, for a formula $\alpha$ in $\bI$, $\top_{\alpha}=\alpha\rightarrow\alpha$, we also have top elements, all equivalent to one another.


\subsection{Bivaluation Semantics for \bI}

A bivaluation for $\bI$ is a map $\nu:\mathbf{F}(\Sigma_{\bI}, \mathcal{V})\rightarrow \{0,1\}$ such that:
\begin{enumerate}
\item $\nu(\alpha\vee\beta)=1\Leftrightarrow\nu(\alpha)=1$ or $\nu(\beta)=1$;
\item $\nu(\alpha\wedge\beta)=1\Leftrightarrow\nu(\alpha)=\nu(\beta)=1$;
\item $\nu(\alpha\rightarrow\beta)=1\Leftrightarrow\nu(\alpha)=0$ or $\nu(\beta)=1$;
\item if $\nu(\alpha\inc\beta)=1$ then $\nu(\alpha)=0$ or $\nu(\beta)=0$;
\item $\nu(\alpha\inc\beta)=\nu(\beta\inc\alpha)$.
\end{enumerate}

Given a set of formulas $\Gamma\cup\{\varphi\}$ of $\bI$, we say that $\varphi$ is a semantical consequence of  $\Gamma$, and write $\Gamma\vDash_{\bI}\varphi$, if for every bivaluation $\nu$ for $\bI$, $\nu(\gamma)=1$ for every $\gamma \in \Gamma$ implies that $\nu(\varphi)=1$. If $\emptyset\vDash_{\bI}\varphi$, we simply write $\vDash_{\bI}\varphi$. We notice, first of all, that if $\Gamma\vDash_{\bI}\alpha$ and $\Gamma\vDash_{\bI}\alpha\rightarrow\beta$, then $\Gamma\vDash_{\bI}\beta$. Furthermore, for any instance $\varphi$ of an axiom schema of $\bI$, one finds that $\vDash_{\bI}\varphi$.

Recall that, given a logic {\bf L}, a set $\Delta$ of formulas is $\varphi$-saturated in {\bf L} if $\Delta \nvdash_{\bf L} \varphi$ but   $\Delta,\psi \vdash_{\bf L} \varphi$ for any $\psi \notin\Delta$. It is well known that, if {\bf L} is  Tarskian and finitary and $\Gamma \nvdash \varphi$, then there exists a set of formulas $\Delta$ such that $\Gamma \subseteq \Delta$ and $\Delta$ is $\varphi$-saturated in {\bf L} (see, for instance,  \cite[Theorem 22.2]{wojcicki1984lectures}). In particular, this result holds for \bI. Any $\varphi$-saturated set $\Delta$ is a closed theory in {\bf L}, that is: $\Delta \vdash_{\bf L} \psi$ iff $\psi \in \Delta$. It is easy to prove that if $\Delta$ is $\varphi$-saturated in \bI\ then the function $\nu$, from the set of formulas of $\bI$ to $\{0,1\}$ and such that $\nu(\gamma)=1$ if and only if $\gamma\in\Delta$, is a bivaluation for $\bI$. Using this we prove the following:

\begin{teo} [Soundness and Completeness of \bI\ w.r.t. bivaluations] 
Given formulas $\Gamma\cup\{\varphi\}$ in $\bI$, $\Gamma\vdash_{\bI}\varphi$ if and only if $\Gamma\vDash_{\bI}\varphi$.
\end{teo}

\begin{proof} \ \\
({\em `Only if'} direction) Proceed by induction on the length $n$ of a proof $\alpha_{1}, \ldots , \alpha_{n}=\varphi$ in \bI\ of $\varphi$ from $\Gamma$ to show that, if $\nu(\gamma)=1$ for every $\gamma \in \Gamma$, then $\nu(\alpha_{1})=\cdots=\nu(\alpha_{n})=1$.\\[1mm]
({\em `Only if'} direction) Suppose by contraposition that $\Gamma\not\vdash_{\bI}\varphi$. Then, there exists a $\varphi$-saturated set $\Delta$ in \bI\ containing  $\Gamma$. The function $\nu$, from the set of formulas of $\bI$ to $\{0,1\}$ such that $\nu(\delta)=1$ if and only if $\delta\in\Delta$, is a bivaluation for \bI. Therefore $\nu$ is a bivaluation such that $\nu(\gamma)=1$ for every $\gamma \in \Gamma$ but $\nu(\varphi)=0$, showing that  $\Gamma\nvDash_{\bI}\varphi$.\qedhere
\end{proof}


\subsection{Restricted Non-deterministic Matrices for \bI}


In~\cite{CostaRNmatrix} we introduce the notion of  restricted Nmatrices (RNmatrices), a semantical framework which generalizes Avron and Lev's non-deterministic matrices (or Nmatrices) proposed in~\cite{avr:lev:01}. These notions will be briefly recalled  below.

Let $\Theta$ be a propositional signature. A {\em logical matrix} over $\Theta$ is a pair $(\A,D)$ such that $\A=(A,\{\sigma_\A\}_{\sigma \in \Theta})$ is an algebra over $\Theta$ and $D$ is a proper non-empty subset of $A$. A {\em valuation} over a logical matrix is a homomorphism $\nu:\mathbf{F}(\Theta, \mathcal{V}) \to \A$.

A {\em non-deterministic matrix} (or {\em Nmatrix})  over $\Theta$ is a pair $(\A,D)$ such that $\A=(A,\{\sigma_\A\}_{\sigma \in \Theta})$ is a multialgebra (or hyperalgebra) over $\Theta$ (that is, $\sigma_\A:A^n \to \wp(A)\setminus \{\emptyset\}$ for every $n$-ary connective $\sigma$) and $D$ is a proper non-empty subset of $A$. A {\em valuation} over a Nmatrix is a function $\nu:\mathbf{F}(\Theta, \mathcal{V}) \to A$ such that $\nu(\sigma(\varphi_1,\ldots,\varphi_n)) \in \sigma_\A(\nu(\varphi_1), \ldots \nu(\varphi_n))$.

A {\em restricted non-deterministic matrix} ({\em restricted Nmatrix}, or {\em RNmatrix} in short) is a triple $(\A,D,\F)$ such that $(\A,D)$ is a Nmatrix over $\Theta$ and $\F$ is a non-empty set of valuations over it. A RNmatrix is said to be {\em structural} if $\nu \circ \lambda \in \F$ for every $\nu \in \F$ and every substitution $\lambda$ over $\Theta$ (that is, every endomorphism $\lambda:\mathbf{F}(\Theta, \mathcal{V}) \to \mathbf{F}(\Theta, \mathcal{V})$).

Let \M\ be a logical matrix, a Nmatrix or a RNmatrix over $\Theta$. Given $\Gamma\cup\{\varphi\} \subseteq \mathbf{F}(\Theta, \mathcal{V})$  we say that $\varphi$ follows from  $\Gamma$  according to \M, and write $\Gamma\vDash_{\M}\varphi$, if for every valuation $\nu$ over \M\ (for every $\nu\in\F$ if \M\ is a RNmatrix), $\nu(\gamma)\in D$ for every $\gamma \in \Gamma$ implies  that $\nu(\varphi)\in D$. These notions give origin to Tarskian and structural logics (in the case of RNmatrices, the RNmatrix must be structural to produce a structural logic).

Now, a RNmatrix semantics for \bI\ will be proposed. Recall first that a {\em classical implicative lattice} is the reduct of a Boolean algebra (defined over the signature $\{\land,\lor,\to,\bot\}$) to the signature $\Sigma^{\cplp}=\{\land,\lor,\to\}$. These structures are the algebraic semantics of \cplp. 
Observe that the signature $\Sigma^{\cplp}$ is obtained from $\Sigma_{\bI}$ by removing the incompatibility connective \inc.

\begin{rem} \label{F-style}
Let  $\A=(A,\{\sigma_\A\}_{\sigma \in \Theta})$ be a  multialgebra. It is worth noting that a multioperator $\sigma_\A:A^n \to \wp(A)\setminus \{\emptyset\}$ such that every set  $\sigma_\A(a_1, \ldots a_n)$ is a singleton can be seen as an ordinary (deterministic) operator $\sigma_\A:A^n \to A$. In particular, if $\sigma_\A$ is a constant (i.e., $\sigma$ is a $0$-ary connective) which is deterministic in \A\ then $\sigma_\A$ can be seen as an element of $A$. Thus, if every $\sigma_\A$ is  deterministic then \A\ can be seen as an ordinary algebra. In general, it could be expected that some operators in a given multialgebra \A\ are deterministic and some others are not. In this case, \A\ has a reduct which is an ordinary algebra.  Notorious examples of this kind of multiagebras are the (nowadays known as) {\em Fidel structures} (or {\bf F}-structures), introduced by Fidel in~\cite{fid:77} to show that da Costa's paraconsistent systems $C_n$ are decidable. As observed in~\cite[Subsection~5.1]{CostaRNmatrix}, {\bf F}-structures constitute a pioneering example of multialgebras formed by ordinary algebras  expanded with some non-deterministic operators (in his case, Boolean algebras expanded with two unary multioperators). When the set of valuations is also considered, {\bf F}-structures for $C_n$ constitute a family of RNmatrices.
\end{rem}

Because of Remark~\ref{F-style}, a multialgebra, Nmatrix or a RNmatrix having as a reduct an ordinary algebra will be called a {\em Fidel-style multialgebra} ({\em Fidel-style (R)Nmatrix}, respectively).

\begin{mydef} \label{RN4bI}
A  RNmatrix for $\bI$ is a triple $(\A,D,\F)$ where
\begin{enumerate}
\item $\mathcal{A}=(A, \{\sigma_{\mathcal{A}}\}_{\sigma\in\Sigma_{\bI}})$ is a (Fidel-style) $\Sigma_{\bI}$-multialgebra  such that:
\begin{enumerate}
\item the $\Sigma^{\cplp}$-reduct  $(A, \{\sigma_{\mathcal{A}}\}_{\sigma\in\Sigma^{\cplp}})$ is a classical implicative lattice having a bottom element $0_\A$, and so it is a Boolean algebra;\footnote{In this case the Boolean complement $\sim$ is defined as usual: $\sneg a= a \to_\A 0_\A$.}
\item  for all $a, b\in A$: $\inc_{\mathcal{A}}(a,b) = \inc_{\mathcal{A}}(b,a)$;
\item for all $a, b,c\in A$: if $c\in\inc_{\mathcal{A}}(a,b)$ then  $\wedge_{\mathcal{A}}(\wedge_{\mathcal{A}}(a,b),c)=0_\A$.
\end{enumerate}

\item $D=\{1_\A\}$.
\item The set \F\ is formed by the valuations $\nu$ over \A\ (i.e., $\nu:\mathbf{F}(\Sigma_{\bI}, \mathcal{V}) \to  \A$  is a $\Sigma_{\bI}$-homomorphism) such that $\nu(\alpha\inc\beta)=\nu(\beta\inc\alpha)$,
for any two formulas $\alpha$ and $\beta$ in $\mathbf{F}(\Sigma_{\bI}, \mathcal{V})$.
\end{enumerate}
\end{mydef}

For simplicity,  the subscript \A\ can be omitted when there is no risk of confusion, and we will use the standard infix notation for binary operators in a  Boolean algebra. Hence, the equation in item~1(c) can be written as  $a \wedge_{\mathcal{A}} b \wedge_{\mathcal{A}}c=0_\A$. The consequence relation w.r.t. RNmatrices for \bI\ will be denoted by $\Vdash_{\mathcal{F}}^{\bI}$. 

\begin{rem} Let $\nu \in \F$ and $\lambda$ be a substitution over $\Sigma_{\bI}$. Then, $\nu\circ\lambda$ is a valuation over \A. Moreover, for any formulas $\alpha$ and $\beta$,
\[\nu\circ\lambda(\alpha \inc\beta)=\nu(\lambda(\alpha) \inc \lambda(\beta)) = \nu(\lambda(\beta) \inc \lambda(\alpha)) = \nu\circ\lambda(\beta \inc\alpha)\]
and so $\nu\circ\lambda \in \F$. That is, any RNmatrix for \bI\ is structural.
\end{rem}

\begin{teo} [Soundness of \bI\ w.r.t. RNmatrices]
Given formulas $\Gamma\cup\{\varphi\}$ of $\bI$, if $\Gamma\vdash_{\bI}\varphi$ then $\Gamma\Vdash_{\mathcal{F}}^{\bI}\varphi$.
\end{teo}
\begin{proof}
Show that, for any instance of an axiom $\alpha$ of $\bI$, $\Vdash_{\mathcal{F}}^{\bI}\alpha$, and that if $\Gamma\Vdash_{\mathcal{F}}^{\bI}\alpha$ and $\Gamma\Vdash_{\mathcal{F}}^{\bI}\alpha\rightarrow\beta$, then $\Gamma\Vdash_{\mathcal{F}}^{\bI}\beta$. Proceed then by induction on the length $n$ of a proof $\alpha_{1}, \ldots  , \alpha_{n}=\varphi$ in \bI\ of $\varphi$ from $\Gamma$ to prove that, if $\nu(\Gamma)\subseteq\{1\}$, then $\nu(\alpha_{i})=1$ for every $i\in\{1, \ldots  , n\}$.\qedhere
\end{proof}

To show completeness of \bI\ w.r.t. RNmatrices we define, for a set of formulas $\Gamma$ of \bI, the following relation $\equiv_{\Gamma}^{\bI}$ between formulas: $\alpha\equiv_{\Gamma}^{\bI}\beta$ iff $\quad\Gamma\vdash_{\bI}\alpha\rightarrow\beta$ and $\Gamma\vdash_{\bI}\beta\rightarrow\alpha$. It is straightforward to prove the following:

\begin{prop}
$\equiv_{\Gamma}^{\bI}$ is a congruence with respect to any $\#\in\{\vee, \wedge, \rightarrow\}$. That is, $\equiv_{\Gamma}^{\bI}$ is an equivalence relation such that $\alpha_{1}\equiv_{\Gamma}^{\bI}\beta_{1}$ and $\alpha_{2}\equiv_{\Gamma}^{\bI}\beta_{2}$ imply $\alpha_{1}\#\beta_{1}\equiv_{\Gamma}^{\bI}\alpha_{2}\#\beta_{2}$.
\end{prop}

Let $A^{\bI}_{\Gamma}=\mathbf{F}(\Sigma_{\bI}, \mathcal{V})/\equiv_{\Gamma}^{\bI}$ be  the quotient set, where $[\alpha]$ will denote the equivalence class of $\alpha$. From the fact that $\equiv_{\Gamma}^{\bI}$ is a congruence with respect to $\#\in\{\vee, \wedge, \rightarrow\}$ we get that $[\alpha]\#_{\mathcal{A}}[\beta]=[\alpha\#\beta]$ are well-defined  operations. Let $0=[\bot_{\alpha\beta}]$ and $1=[\top_{\alpha}]$ (both independent of the chosen $\alpha$ and $\beta$) and ${\sim}[\alpha]=[\alpha\rightarrow\bot_{\alpha\alpha}]$. Then, the following holds (to see a proof for the case of \mbc\ which is similar to the present one, we refer back to~\cite[Proposition~6.1.7]{ParLog}).

\begin{prop}
$\mathcal{A}=(A^{\bI}_{\Gamma}, \{\sigma_{\mathcal{A}}\}_{\sigma\in\Sigma^{\cplp}})$ is a Boolean algebra with bottom $[\bot_{\alpha\beta}]$ and top $[\top_{\alpha}]$.
\end{prop}

For a given $\Gamma$, we can expand $\mathcal{A}$ with a multioperator \inc\ in order to get a RNmatrix for  $\bI$ by defining
\[[\alpha]\inc[\beta]=\{[\varphi\inc\psi] \ : \  \varphi\in [\alpha] \ \mbox{ and } \ \psi\in[\beta]\}.\]
This produces a $\Sigma_{\bI}$-multialgebra $\mathcal{A}^{\bI}_{\Gamma}$, with universe $A^{\bI}_{\Gamma}$, called {\em the Lindenbaum-Tarski multialgebra of $\bI$ associated to $\Gamma$}. We can also prove that $[\alpha]\inc[\beta]=[\beta]\inc[\alpha]$. This follows from the fact that, from $\textbf{Comm}$, $\Gamma\vdash_{\bI}(\psi\inc\varphi)\rightarrow(\varphi\inc\psi)$ and $\Gamma\vdash_{\bI}(\varphi\inc\psi)\rightarrow(\psi\inc\varphi)$, meaning that $[\psi\inc\varphi]=[\varphi\inc\psi]$ for every $\psi, \varphi$. By taking $D$ and \F\ over $\mathcal{A}^{\bI}_{\Gamma}$ as in Definition~\ref{RN4bI}, the induced RNmatrix for $\bI$ will be called {\em the Lindenbaum-Tarski RNmatrix of $\bI$ associated to $\Gamma$}. Using this structure we get the following:

\begin{teo} [Completeness of \bI\ w.r.t. RNmatrices] \label{compleRNmatbI}
Given formulas $\Gamma\cup\{\varphi\}$ of $\bI$, if $\Gamma\Vdash^{\bI}_{\mathcal{F}}\varphi$ then $\Gamma\vdash_{\bI}\varphi$.
\end{teo}
\begin{proof}
Suppose that $\Gamma\nvdash_{\bI}\varphi$, and consider a $\varphi$-saturated set $\Delta$ in \bI\ such that $\Gamma \subseteq \Delta$. Let $\mathcal{A}^{\bI}_{\Delta}$ be the Lindenbaum-Tarski multialgebra of $\bI$ associated to $\Delta$, and consider the RNmatrix generated from this, as defined above. Since $\Delta$ is $\varphi$-saturated, the map $\nu:\mathbf{F}(\Sigma_{\bI}, \mathcal{V})\rightarrow A^{\bI}_{\Delta}$ such that $\nu(\alpha)=[\alpha]$ is a valuation for $\bI$ over $\mathcal{A}^{\bI}_{\Delta}$ which clearly is in  $\mathcal{F}$. Furthermore, $\nu(\alpha)=1$ if and only if $\Delta\vdash_{\bI}\alpha$. From this, $\nu(\gamma)=1$ for every $\gamma \in \Gamma$. Given that $\Delta\nvdash_{\bI}\varphi$, we obtain that $\nu(\varphi)=0$. This shows that  $\Gamma\nVdash^{\bI}_{\mathcal{F}}\varphi$.\qedhere
\end{proof}


\subsection{A Decision Method for \bI}

We will denote by $\textbf{2}$ the two-valued Boolean algebra with domain $\{0,1\}$. We define a RNmatrix  for \bI\ over $\textbf{2}$ as follows: let $\textbf{2}_{\bI}$ be the expansion of the $\Sigma^{\cplp}$-reduct of $\textbf{2}$ with the multioperator \inc\ given by $1\inc1=\{0\}$ and $x\inc y=\{0,1\}$ otherwise. Observe that, for every $x,y,z$, if  $z\in x\inc y$ then $x\wedge y\wedge z=0$. Moreover, $x\inc y=y\inc x$.

\begin{figure}[H]
\centering
\captionsetup{justification=centering}
\begin{tabular}{|l|c|r|}
\hline
$\inc$ & $0$ & $1$\\ \hline
$0$ & $\{0,1\}$ & $\{0,1\}$\\ \hline
$1$ & $\{0,1\}$ & $\{0\}$\\\hline
\end{tabular}
\caption*{Table for $\inc$ in $\textbf{2}_{\bI}$}
\end{figure}
Let $\M^\textbf{2}_{\bI}=(\textbf{2}_{\bI}, \{1\}, \mathcal{F}_{\textbf{2}_{\bI}})$  be the RNmatrix for \bI\  defined from $\textbf{2}_{\bI}$ according to Definition~\ref{RN4bI}.

\begin{prop}\label{bivaluations are homomorphisms}
A map $\nu:\mathbf{F}(\Sigma_{\bI}, \mathcal{V})\rightarrow\{0,1\}$ is a bivaluation for $\bI$ if, and only if, it is a valuation for the Nmatrix $(\textbf{2}_{\bI}, \{1\})$ which lies in $\mathcal{F}_{\textbf{2}_{\bI}}$.
\end{prop}

\begin{cor}
Given formulas $\Gamma\cup\{\varphi\}$ of $\bI$, $\Gamma\vDash_{\bI}\varphi$ iff $\Gamma\vDash_{\M^\textbf{2}_{\bI}}\varphi$. Hence, $\Gamma\vdash_{\bI}\varphi$ iff $\Gamma\vDash_{\M^\textbf{2}_{\bI}}\varphi$.
\end{cor}

The RNmatrix $\M^\textbf{2}_{\bI}$ induces a straightforward decision method for \bI. Here, we give a brief description of the technique, the proofs this works being easily adapted from the previously mentioned~\cite{CostaRNmatrix}.

In order to test whether a formula $\varphi$ of $\bI$ is a tautology, one starts by writing its row-branching table according to the Nmatrix $(\textbf{2}_{\bI}, \{1\})$. This involves listing the subformulas of $\varphi$ in ascending degree of complexity\footnote{The {\em complexity} of a formula in $\bI$ is defined recursively as usual: variables have complexity $0$, and if $\alpha$ and $\beta$ have complexity $m$ and $n$, respectively, then $\alpha\#\beta$ has complexity $m+ n+1$, for $\#\in\{\vee, \wedge, \rightarrow, \inc\}$.} $\varphi_{1}$, \ldots  , $\varphi_{n}=\varphi$, and listing all the possible values for those subformulas that are variables (the possible values being $0$ or $1$, independently of the values of other subformulas). Subsequently, if $\varphi_{l}=\varphi_{i}\#\varphi_{j}$, and $\varphi_{i}$ assumes the value $a$ and $\varphi_{j}$ assumes the value $b$ on a given row, then $\varphi_{l}$ assumes the value $a\#b$ on the same row if $\#\in\{\vee, \wedge, \rightarrow\}$. If $\#$ is $\inc$ and $a=b=1$, $\varphi_{l}$ assumes the value $0$, and if $\#$ is $\inc$ but either $a$ or $b$ is $0$, the row in question branches into two, one assigning the value $0$ to $\varphi_{l}$, the other the value $1$.

Finally, the rows corresponding to undesired homomorphisms must be erased. If a row contains both $\varphi_{i}\inc\varphi_{j}$ and $\varphi_{j}\inc\varphi_{i}$ and they are given different values, the row must be erased. This can of course be done algorithmically, and then $\varphi$ is a tautology iff its corresponding column on the  table reduced as described above contains only $1$. To test a deduction $\Gamma\vdash_{\bI}\varphi$, with $\Gamma=\{\gamma_{1}, \ldots  , \gamma_{m}\}$ a finite set, it is enough to test if $\bigwedge_{i=1}^{m}\gamma_{i}\rightarrow\varphi$ is a tautology. Thus, the row-branching, row-eliminating tables for $\bI$ are a decision method for this logic. Notice too that it is clearly more efficient to erase the undesired rows as they appear: if $\varphi_{i}$ takes the value $a$ and $\varphi_{j}$ the value $b$ with $a=0$ or $b=0$, and $\varphi_{k}=\varphi_{j}\inc\varphi_{i}$ has already appeared, then $\varphi_{l}=\varphi_{i}\inc\varphi_{j}$ simply takes the same value as $\varphi_{k}$.

\subsection{A Tableaux Decision Method for \bI}

In our previous article \cite{CostaRNmatrix} we constructed labelled tableau calculi for the logics $C_{n}$ of da Costa based on their corresponding RNmatrices. This can again be done here, but we will not delve into details. The labelled tableau calculus for $\bI$, which we will denote by $\mathbb{T}_{\bI}$, has the following rules, being the ones for the classical connectives the expected ones.\footnote{Notice that by adding a classical negation $\sim$ to $\bI$ we could avoid the use of labels.}

$$
\begin{array}{cp{1.5cm}cp{1.5cm}c}
\displaystyle \frac{\textsf{0}(\varphi\vee\psi)}{\begin{array}{c}\textsf{0}(\varphi) \\ \textsf{0}(\psi)\end{array}} & & \displaystyle \frac{\textsf{0}(\varphi\wedge\psi)}{\textsf{0}(\varphi)\mid\textsf{0}(\psi)} & & \displaystyle \frac{\textsf{0}(\varphi\rightarrow\psi)}{\begin{array}{c}\textsf{1}(\varphi) \\ \textsf{0}(\psi)\end{array}}  \\[2mm]
&&&&\\[2mm]
 \displaystyle \frac{\textsf{1}(\varphi\vee\psi)}{\textsf{1}(\varphi)\mid\textsf{1}(\psi)} & & \displaystyle \frac{\textsf{1}(\varphi\wedge\psi)}{\begin{array}{c}\textsf{1}(\varphi) \\ \textsf{1}(\psi)\end{array}} & & \displaystyle \frac{\textsf{1}(\varphi\rightarrow\psi)}{\textsf{0}(\varphi)\mid\textsf{1}(\psi)}\\[2mm]
&&&&\\[2mm]
\end{array}
$$

\[\frac{\textsf{1}(\varphi\inc\psi)}{\textsf{0}(\varphi)\mid\textsf{0}(\psi)}\]

A branch of a tableau in $\mathbb{T}_{\bI}$ is closed if:
\begin{enumerate}
\item it contains labelled formulas $\textsf{L}(\varphi)$ and $\textsf{L}'(\varphi)$ with $\textsf{L}\neq\textsf{L}'$;
\item it contains labelled formulas $\textsf{L}(\varphi\inc\psi)$ and $\textsf{L}'(\psi\inc\varphi)$ with $\textsf{L}\neq\textsf{L}'$.
\end{enumerate}
A branch $\theta$ is complete if, for every labelled formula $\textsf{L}(\gamma)$ in $\theta$ not of the form $\textsf{0}(\varphi\inc\psi)$ (for which there is no tableau rule) and with $\gamma$ not a variable, $\theta$ also contains all of the labelled formulas of one of the branches resulting from the application of a tableau rule to $\textsf{L}(\gamma)$ (there is only one applicable rule). A complete branch is open if it is not closed. A tableau in $\mathbb{T}_{\bI}$ is: closed if all of its branches are closed; complete if all of its branches are either closed or complete; and open if its complete but not closed.

Given that all the tableau rules are analytic, in the sense that the formulas in the conclusion are of complexity strictly smaller than that of the premiss (indeed, they are strict subformulas of it), the resulting tableaux of $\mathbb{T}_{\bI}$ can always be completed. A formula $\varphi$ of $\bI$ is said to be provable by tableaux in $\mathbb{T}_{\bI}$, what we write as $\vdash_{\mathbb{T}_{\bI}}\varphi$, if there exists a closed tableau in $\mathbb{T}_{\bI}$ starting from $\textsf{0}(\varphi)$. If $\Gamma$ is a finite set of formulas $\{\gamma_{1}, \ldots  , \gamma_{m}\}$, we say that $\varphi$ is provable by tableaux in $\mathbb{T}_{\bI}$ from $\Gamma$, written as $\Gamma\vdash_{\mathbb{T}_{\bI}}\varphi$, if there is a closed tableau in $\mathbb{T}_{\bI}$ starting from $\bigwedge_{i=1}^{m}\gamma_{i}\rightarrow\varphi$. The following theorem finally shows how the tableau system become a decision method for $\bI$.

\begin{teo}
For a finite set of formulas $\Gamma\cup\{\varphi\}$, $\Gamma\vdash_{\bI}\varphi$ iff $\Gamma\vdash_{\mathbb{T}_{\bI}}\varphi$.
\end{teo}


\subsection{Collapsing axioms for \bI}

A question that naturally arises when dealing with incompatibility is the relationship between $\alpha$ and $\beta$ being incompatible, and $\alpha$ and $\beta$, together, trivializing the logic. That is, if $\alpha$ and $\beta$ can trivialize the given logic, does that mean $\alpha$ and $\beta$ are incompatible? Consider the axiom schema
\[\tag{\textbf{Ex}}((\alpha\wedge\beta)\rightarrow\bot_{\alpha\beta})\rightarrow(\alpha\inc\beta)\]
The logic obtained from $\bI$ by addition of $\textbf{Ex}$ is not really a new logic: in this system $\alpha\inc\beta$ is equivalent to $(\alpha\wedge\beta)\rightarrow\bot_{\alpha\beta}$ and therefore we reobtain $\textbf{CPL}$ in which $\alpha \inc \beta$ is  $\alpha{\Uparrow}\beta$.

The axiom $\textbf{ciw}$, given by $\circ\alpha\vee(\alpha\wedge\neg \alpha)$, is a very important one when dealing with paraconsistency in \lfis. When dealing with incompatibility, we can consider an analogous axiom, namely
\[\tag{$\textbf{ciw}^\inc$}(\alpha\inc\beta)\vee(\alpha\wedge\beta).\]
The logic obtained from $\bI$ by adding $\textbf{ciw}^\inc$ is, as it happens with $\bI$ and $\textbf{Ex}$,  equivalent to $\textbf{CPL}$, with $\alpha\inc\beta$ once again corresponding to $\alpha{\Uparrow}\beta$. Both $\textbf{Ex}$ and $\textbf{ciw}^{\inc}$ are then  collapsing axioms: although very intuitive, their addition to $\bI$ takes us back to $\textbf{CPL}$, showing how very sensitive this former system can be. In the next section, some possible extensions of \bI\ by means of a (non-classical) negation will be investigated.


\section{Adding a Negation to Logics of Incompatibility}\label{third}

When studying paraconsistent logics, our main focus is in the properties of negations, as is the case when working with paracomplete logics (in which the given negation does not satisfy the excluded middle). Given such a prominent role negation plays in non-classical logics, is quite natural to shift our focus in the logics of incompatibility from $\inc$ to a non-classical negation, or better yet, to the possible interplay between $\inc$ and such a negation.

Our first step is adding this negation: we have seen that in any logic extending $\bI$, is always possible to define a classical negation. We, therefore, need to consider a new negation, weaker than the negation intrinsic to $\bI$. So we need a symbol for it: we define the signature $\Sigma_{\nbI}$ as the signature obtained from $\Sigma_{\bI}$ by addition of a unary symbol $\neg$.


\subsection{The Logic $\nbI$}

We start by adding to $\bI$ a paraconsistent negation, producing the logic $\nbI$. To the Hilbert calculus for $\bI$ we add only the axiom schema
\begin{enumerate}[wide=0pt, leftmargin=*]
\item[\textbf{Ax  11}] $\alpha\vee\neg \alpha$.
\end{enumerate}

A bivaluation for $\nbI$ is a map $\nu:\mathbf{F}(\Sigma_{\nbI},\mathcal{V})\rightarrow\{0,1\}$ satisfying the conditions 1-5 required of a bivaluation for $\bI$, plus the following:
\begin{itemize}
\item[6.] if $\nu(\neg \alpha)=0$, then $\nu(\alpha)=1$.
\end{itemize}
We define the semantical consequence relation $\vDash_{\nbI}$ in the same way we had defined $\vDash_{\bI}$.
Clearly, \nbI-bivaluations validate \textbf{Ax  11}. From this, we obtain easily the following: 

\begin{teo} [Soundness of \nbI\ w.r.t. bivaluations]
Given formulas $\Gamma\cup\{\varphi\}$ of $\nbI$, if $\Gamma\vdash_{\nbI}\varphi$ then $\Gamma\vDash_{\nbI}\varphi$.
\end{teo}

From the latter result it follows that \nbI\ is paraconsistent. Indeed,  if $p$ and $q$ are two  different variables, consider a \nbI-valuation $\nu$ such that $\nu(p)=\nu(\neg p)=1$ and $\nu(q)=0$. This shows that  $p,\neg p \nvDash_{\nbI} q$ and so  $p,\neg p \nvdash_{\nbI} q$, by soundness.

In order to prove completeness, it is immediate to see that, given a set of formulas $\Gamma$ which is  $\varphi$-saturated in $\nbI$, the map $\nu:\mathbf{F}(\Sigma_{\nbI}, \mathcal{V})\rightarrow\{0,1\}$, such that $\nu(\gamma)=1$ if and only if $\gamma\in\Gamma$, is a bivaluation for $\nbI$. Then, as in the case of \bI, it can be proved by contraposition the following:

\begin{teo}  [Completeness of \nbI\ w.r.t. bivaluations]
Given formulas $\Gamma\cup\{\varphi\}$ of $\nbI$, if $\Gamma\vDash_{\nbI}\varphi$ then $\Gamma\vdash_{\nbI}\varphi$.
\end{teo}


\subsection{RNmatrices for \nbI}

The RNmatrices for \bI\ can be easily extended to RNmatrices for \nbI.

\begin{mydef} \label{defRNmat-nbI}
A  RNmatrix for $\nbI$ is a triple $(\A,D,\F)$ where
\begin{enumerate}
\item $\mathcal{A}=(A, \{\sigma_{\mathcal{A}}\}_{\sigma\in\Sigma_{\nbI}})$ is a $\Sigma_{\nbI}$-multialgebra  such that the reduct $(A,\{\sigma_{\mathcal{A}}\}_{\sigma\in\Sigma_{\bI}})$ satisfies the conditions of item~1 of Definition~\ref{RN4bI}.
\item $\neg_\A a \subseteq \{b\in A \ : \ a\vee_\A b=1_\A\}$, for all $a\in A$.
\item $D=\{1_\A\}$.
\item The set \F\ is formed by the valuations $\nu$ over \A\ (i.e., $\nu:\mathbf{F}(\Sigma_{\nbI}, \mathcal{V}) \to  \A$  is a $\Sigma_{\nbI}$-homomorphism) such that $\nu(\alpha\inc\beta)=\nu(\beta\inc\alpha)$,
for any two formulas $\alpha$ and $\beta$ in $\mathbf{F}(\Sigma_{\nbI}, \mathcal{V})$.
\end{enumerate}
\end{mydef}

Clearly, any RNmatrix for \nbI\ is structural (the proof is similar to the one for \bI).
The consequence relation w.r.t. RNmatrices for \nbI\ will be denoted by $\Vdash_{\mathcal{F}}^{\nbI}$.

Since RNmatrices for $\nbI$ satisfy all the properties that the ones for $\bI$ have, they validate  the axiom schemata and rules of inference of $\bI$. Furthermore, for any RNmatrix $\mathcal{A}$ for $\nbI$ and valuation $\nu$, $\nu(\alpha\vee\neg \alpha)=\nu(\alpha)\vee_\A\nu(\neg \alpha)$, and since $\nu(\neg \alpha)\in\neg_\A\nu(\alpha)$, $\nu(\alpha)\vee_\A\nu(\neg \alpha)=1_\A$. From this we get the following:

\begin{teo} [Soundness of \nbI\ w.r.t. RNmatrices]
Given formulas $\Gamma\cup\{\varphi\}$ of $\nbI$, if $\Gamma\vdash_{\nbI}\varphi$ then $\Gamma\Vdash_{\mathcal{F}}^{\nbI}\varphi$.
\end{teo}

To prove completeness, we define the equivalence relation, for a fixed set of formulas $\Gamma$, between formulas of $\nbI$ such that $\alpha\equiv_{\Gamma}^{\nbI}\beta$ iff $\Gamma\vdash_{\nbI}\alpha\rightarrow\beta$ and $\Gamma\vdash_{\nbI}\beta\rightarrow\alpha$. This relation is a congruence with respect to the connectives in $\Sigma^{\cplp}$, and then $A^{\nbI}_{\Gamma}=\mathbf{F}(\Sigma_{\nbI},\mathcal{V})/\equiv_{\Gamma}^{\nbI}$ becomes  a Boolean algebra (as in the case of \bI). By defining, for classes of formulas $[\alpha]$ and $[\beta]$, $\neg[\alpha]=\{[\neg \varphi] \ : \  \varphi\in[\alpha]\}$ and $[\alpha]\inc[\beta]$ as in $\mathcal{A}^{\bI}_{\Gamma}$, we obtain a $\Sigma_{\nbI}$-multialgebra $\mathcal{A}^{\nbI}_{\Gamma}$, which we shall call {\em the Lindenbaum-Tarski multialgebra of $\nbI$ associated to $\Gamma$}.
It is easy to see that the reduct of $\mathcal{A}^{\nbI}_{\Gamma}$ to $\Sigma_{\bI}$ satisfies the conditions of item~1 of Definition~\ref{RN4bI}. In addition, proving that $[\alpha]\vee[\beta]=1$, for any $[\beta]\in \neg[\alpha]$, is also immediate. By taking $D$ and \F\ over $\mathcal{A}^{\bI}_{\Gamma}$ as in Definition~\ref{defRNmat-nbI}, the induced RNmatrix for $\nbI$ will be called {\em the Lindenbaum-Tarski RNmatrix of $\nbI$ associated to $\Gamma$}. From this we prove, by adapting the proof of Theorem~\ref{compleRNmatbI}, the following result:

\begin{teo} [Completeness of \nbI\ w.r.t. RNmatrices] \label{complRNnbI}
Given formulas $\Gamma\cup\{\varphi\}$ of $\nbI$, if $\Gamma\Vdash_{\mathcal{F}}^{\nbI}\varphi$ then $\Gamma\vdash_{\nbI}\varphi$.
\end{teo}


\subsection{Two Decision Methods for \nbI} \label{dec-nbI}

Take the Boolean algebra $\textbf{2}$ again and expand the $\Sigma_{\bI}$-multialgebra $\textbf{2}_{\bI}$ with  an unary multioperation $\neg$ such that  $\neg 0=\{1\}$ and $\neg 1=\{0,1\}$. With these multioperations $\textbf{2}$ becomes a $\Sigma_{\nbI}$-multialgebra $\textbf{2}_{\nbI}$ which satisfies the conditions of item~1 of Definition~\ref{RN4bI}. Furthermore, for any $y\in\neg x$, $x\vee y=1$. 
Let $\M^\textbf{2}_{\nbI}=(\textbf{2}_{\nbI}, \{1\}, \mathcal{F}_{\textbf{2}_{\nbI}})$  be the RNmatrix for \nbI\  defined from $\textbf{2}_{\nbI}$ according to Definition~\ref{defRNmat-nbI}. The proof of the next results is immediate:

\begin{prop}
A map $\nu:\mathbf{F}(\Sigma_{\nbI},\mathcal{V})\rightarrow\{0,1\}$ is a bivaluation for $\nbI$ if, and only if, it is a valuation for the Nmatrix $(\textbf{2}_{\nbI}, \{1\})$ which lies in $\mathcal{F}_{\textbf{2}_{\nbI}}$.
\end{prop}

\begin{cor}
Given formulas $\Gamma\cup\{\varphi\}$ of $\nbI$, $\Gamma\vDash_{\nbI}\varphi$ iff $\Gamma\vDash_{\M^\textbf{2}_{\nbI}}\varphi$. Hence, $\Gamma\vdash_{\nbI}\varphi$ iff $\Gamma\vDash_{\M^\textbf{2}_{\nbI}}\varphi$.
\end{cor}

By writing the row-branching table for a formula $\varphi$ according to the Nmatrix $(\textbf{2}_{\nbI}, \{1\})$ and erasing the rows where $\alpha\inc\beta$ and $\beta\inc\alpha$ receive different values, we obtain a decision method for $\nbI$ based on row-branching, row-eliminating truth tables.

A second decision procedure can be obtained by tableaux.
By adding the  rule
\[\frac{\textsf{0}(\neg\varphi)}{\textsf{1}(\varphi)}\]
to the labelled tableau calculus $\mathbb{T}_{\bI}$, we obtain a tableau calculus  $\mathbb{T}_{\nbI}$ for $\nbI$ which is again sound and complete, and therefore constitutes another decision method for this logic. Observe that the notion of complete branch of a tableau in  $\mathbb{T}_{\nbI}$ does not need to consider, besides the labelled formulas $\textsf{0}(\varphi\inc\psi)$ and  $\textsf{L}(\gamma)$ for $\gamma$ a variable, labelled formulas of the form $\textsf{1}(\neg\varphi)$   (since there are no rules to apply to them).


\subsection{Collapsing axioms for \nbI}

Now that we have a paraconsistent negation, we are capable of doing to $\textbf{ci}$ and $\textbf{cl}$ the same we did to $\textbf{ciw}$ when we transformed it into $\textbf{ciw}^{\inc}$. Thus, consider the following axiom schemas:
\[\tag{$\textbf{ci}^{\inc}$}\neg(\alpha\inc\beta)\rightarrow(\alpha\wedge\beta)\]
\[\tag{$\textbf{cl}^{\inc}$}\neg(\alpha\wedge\beta)\rightarrow(\alpha\inc\beta).\]
Adding to $\nbI$ any of the two axioms collapses the system to $\textbf{mbC}$ (i.e., $\textbf{CPL}$ with a paraconsistent negation), with $\alpha\inc\beta$ being equivalent to $\alpha\wedge\beta\rightarrow\bot_{\alpha\beta}$, that is, $\alpha{\Uparrow}\beta$.

One can actually prove a stronger assertion: instances of $\textbf{ci}^{\inc}$ and $\textbf{cl}^{\inc}$ actually imply their corresponding instances of $\textbf{ciw}^{\inc}$ in $\nbI$, that is,  each of $\neg(\alpha\inc\beta)\rightarrow(\alpha\wedge\beta)$ and $\neg(\alpha\wedge\beta)\rightarrow(\alpha\inc\beta)$ implies $(\alpha\inc\beta)\vee(\alpha\wedge\beta)$ in $\nbI$. And given that $\textbf{ciw}^{\inc}$ implies the equivalence between $\alpha\inc\beta$ and $\alpha{\Uparrow}\beta$, certainly each of $\textbf{ci}^{\inc}$ and $\textbf{cl}^{\inc}$ implies it too.


\subsection{Some Extensions of \nbI}

Notice that the basic axiom $\textbf{Ip}$ of $\bI$ is quite similar in its structure to the basic axiom $\textbf{bc1}$ of $\textbf{mbC}$, with $\neg\alpha$ replaced by $\beta$ and $\circ\alpha$ replaced by $\alpha\inc\beta$.

An attempt to achieve something similar with the axiom $\textbf{ciw}$ is not successful, as $\textbf{ciw}^{\inc}$ collapses the incompatibility operator \inc\ with the Sheffer's stroke $\Uparrow$. However, in $\nbI$, where we have at our disposal a non-classical negation $\neg$, one can adapt $\textbf{ciw}$ to the language $\Sigma_{\nbI}$, instead of generalizing it to $\textbf{ciw}^{\inc}$. The same can be done with axioms $\textbf{ci}$ and $\textbf{cl}$ for \lfis, producing three logics that are both logics of incompatibility and formal inconsistency.

So, over the signature $\Sigma_{\nbI}$, we consider the logics $\nbIciw$, $\nbIci$ and $\nbIcl$, obtained from $\nbI$ by addition, respectively,  of the schemas

\[\tag{$\textbf{ciw}^{*}$}(\alpha\inc\neg\alpha)\vee(\alpha\wedge\neg\alpha)\]
\[\tag{$\textbf{ci}^{*}$}\neg(\alpha\inc\neg\alpha)\rightarrow(\alpha\wedge\neg\alpha)\]
\[\tag{$\textbf{cl}^{*}$}\neg(\alpha\wedge\neg\alpha)\rightarrow(\alpha\inc\neg\alpha)\]

\begin{prop} \label{ciw-weaker}
Axiom $\textbf{ciw}^{*}$ is derivable in both $\nbIci$ and $\nbIcl$. Hence, $\nbIciw$ is weaker than both logics $\nbIci$ and $\nbIcl$.
\end{prop}
\begin{proof}
Note that $(\alpha\inc\neg\alpha) \vdash_{\nbIci} (\alpha\inc\neg\alpha)\vee(\alpha\wedge\neg\alpha)$ (just using \cplp) and also $\neg(\alpha\inc\neg\alpha) \vdash_{\nbIci} (\alpha\inc\neg\alpha)\vee(\alpha\wedge\neg\alpha)$, by $\textbf{ci}^{*}$ and \cplp. Then $\vdash_{\nbIci} (\alpha\inc\neg\alpha)\vee(\alpha\wedge\neg\alpha)$, by proof-by-cases and \textbf{Ax 10}. The proof for $\nbIcl$ is analogous.
\end{proof}

\subsection{Bivaluation Semantics for the Extensions of \nbI}

Let $\mathcal{L}$ denote an arbitrary logic among $\nbIciw$, $\nbIci$ and $\nbIcl$.

\begin{mydef}
A bivaluation for $\mathcal{L}$ is a bivaluation for $\nbI$ satisfying:
\begin{enumerate}
\item if $\nu(\alpha\inc\neg\alpha)=0$, then $\nu(\alpha)=\nu(\neg\alpha)=1$;
\item \begin{enumerate}
\item if $\nu(\neg(\alpha\inc\neg\alpha))=1$, then $\nu(\alpha)=\nu(\neg\alpha)=1$ \ (if $\mathcal{L}=\nbIci$);
\item if $\nu(\neg(\alpha\wedge\neg\alpha))=1$, then $\nu(\alpha\inc\neg\alpha)=1$ \ (if $\mathcal{L}=\nbIcl$).
\end{enumerate}
\end{enumerate}
\end{mydef}

The semantical consequence relation $\vDash_{\mathcal{L}}$ with respect to bivaluations for $\mathcal{L}$ is defined as usual. Since bivaluations for $\mathcal{L}$ are bivaluations for $\nbI$ satisfying some additional property, $\vDash_{\mathcal{L}}$ models all the axiom schemas of $\nbI$ as well as \MP. Proving that $\vDash_{\nbIciw}$ also models $\textbf{ciw}^{*}$, and analogously for the other logics, is effortless. This produces the following:

\begin{teo} [Soundness of  $\mathcal{L}$ w.r.t. bivaluation semantics]
Given formulas $\Gamma\cup\{\varphi\}$ over $\Sigma_{\nbI}$, if $\Gamma\vdash_{\mathcal{L}}\varphi$ then $\Gamma\vDash_{\mathcal{L}}\varphi$.
\end{teo}

As a consequence of the previous result, it is easy to prove that $\mathcal{L}$ is still paraconsistent w.r.t. $\neg$: it is enough considering  a bivaluation $\nu$ for $\mathcal{L}$ such that $\nu(p)=\nu(\neg p)=1$ and $\nu(q)=0$, where $p$ and $q$ are two different variables. This shows that $p,\neg p \nvDash_{\mathcal{L}} q$ and then  $p,\neg p \nvdash_{\mathcal{L}} q$, by soundness.

\begin{prop} None of the schemas $\textbf{ci}^{*}$ and $\textbf{cl}^{*}$ can be derived in
$\nbIciw$. Thus, $\nbIci$ and $\nbIcl$ are strictly stronger than $\nbIciw$.
\end{prop}

\begin{proof}
Take a bivaluation $\nu$ for $\nbIciw$ such that:\begin{enumerate}
\item $\nu(\alpha)=0$, $\nu(\neg\alpha)=1$, $\nu(\alpha\inc\neg\alpha)=1$ and $\nu(\neg(\alpha\inc\neg\alpha))=1$. Then, $\nu(\neg(\alpha\inc\neg\alpha)\rightarrow(\alpha\wedge\neg\alpha))=0$. This shows that $\textbf{ci}^{*}$ is not valid in $\nbIciw$ and so it is not derivable in $\nbIciw$, by soundness.
\item $\nu(\alpha)=1$, $\nu(\neg\alpha)=1$ (hence $\nu(\alpha\inc\neg\alpha)=0$), and $\nu(\neg(\alpha\wedge\neg\alpha))=1$. From this, $\nu(\neg(\alpha\wedge\neg\alpha)\rightarrow(\alpha\inc\neg\alpha))=0$.  This shows that $\textbf{cl}^{*}$ is not valid in $\nbIciw$ and so it is not derivable in $\nbIciw$, by soundness.\qedhere
\end{enumerate}
\end{proof}

To prove completeness, take a set of formulas $\Gamma$ which is $\varphi$-saturated in $\mathcal{L}$. By defining $\nu:\mathbf{F}(\Sigma_{\nbI},\mathcal{V})\rightarrow\{0,1\}$ such that $\nu(\gamma)=1$ iff $\gamma\in\Gamma$, it is easy to prove that $\nu$ is a bivaluation for $\mathcal{L}$. Hence:

\begin{teo}  [Completeness of  $\mathcal{L}$ w.r.t. bivaluation semantics]
Given formulas $\Gamma\cup\{\varphi\}$ over $\Sigma_{\nbI}$, if $\Gamma\vDash_{\mathcal{L}}\varphi$ then $\Gamma\vdash_{\mathcal{L}}\varphi$.
\end{teo}

\begin{teo}
$\nbIci$ is equivalent to the system obtained from $\nbIciw$ by adding the axiom schema
\[\tag{$\textbf{cc}^{*}$} (\alpha\inc\neg\alpha)\inc\neg(\alpha\inc\neg\alpha).\]
\end{teo}

\begin{proof}
See Proposition $3.1.10$ of \cite{ParLog} for an equivalent result (by observing that $\textbf{cc}^{*}$ corresponds to the schema $\cons\cons\alpha$ in the signature of \lfis).\qedhere
\end{proof}


\subsection{RNmatrix Semantics for the Extensions of \nbI}

The RNmatrix semantics for \nbI\ can be easily adapted in order to deal with its axiomatic extensions.

\begin{mydef} \label{defRNmat-ext-nbI}
Let  $\mathcal{L} \in \{\nbIciw, \nbIci,\nbIcl\}$. A  RNmatrix for $\mathcal{L}$ is a triple $(\A,D,\F)$ where
\begin{enumerate}
\item $\mathcal{A}=(A, \{\sigma_{\mathcal{A}}\}_{\sigma\in\Sigma_{\nbI}})$ is a $\Sigma_{\nbI}$-multialgebra  for \nbI\ (that is, satisfies the conditions of Definition~\ref{defRNmat-nbI}) such that, in addition,  
$$\mbox{if $b\in\neg a$, then } \ {\sim}(a\wedge b) \in a\inc b.$$
\item \begin{enumerate}
\item If $b\in\neg a$, then $a\wedge b\in\neg(a\inc b)$ (if $\mathcal{L}=\nbIci$).
\item If $b\in\neg a$, then ${\sim}(a\wedge b)\in\neg(a\wedge b)$ (if $\mathcal{L}=\nbIcl$).
\end{enumerate}
\item $D=\{1_\A\}$.
\item The set \F\ is formed by the valuations $\nu$ over \A\ (i.e., $\nu:\mathbf{F}(\Sigma_{\nbI}, \mathcal{V}) \to  \A$  is a $\Sigma_{\nbI}$-homomorphism) such that, for all formulas $\alpha$ and $\beta$:
\begin{enumerate}
\item $\nu(\alpha\inc\beta)=\nu(\beta\inc\alpha)$;
\item $\nu(\alpha\inc\neg\alpha)=\sneg(\nu(\alpha) \land \nu(\neg\alpha))$;
\item $\nu(\neg(\alpha\inc\neg\alpha))=\nu(\alpha)\wedge\nu(\neg\alpha)$ (if $\mathcal{L}=\nbIci$).
\item $\nu(\alpha\inc\neg\alpha)=\nu(\neg(\alpha\wedge\neg\alpha))$ (if $\mathcal{L}=\nbIcl$).
\end{enumerate} 
\end{enumerate}
\end{mydef}

It is immediate to see that the set \F\ above is well-defined\footnote{Just to give an example, consider clause 4(d). Since $\nu$ is a homomorphism, $\nu(\neg\alpha) \in \neg\nu(\alpha)$, $\nu(\alpha\inc\neg\alpha) \in \nu(\alpha)\inc\nu(\neg\alpha)$ and $\nu(\neg(\alpha\wedge\neg\alpha)) \in \neg (\nu(\alpha)\wedge\nu(\neg\alpha))$. By clause~1 we can always choose $\nu(\alpha\inc\neg\alpha)=\sneg(\nu(\alpha)\wedge\nu(\neg\alpha))$, thus satisfying 4(b). By clause 2(b) we can always choose $\nu(\neg(\alpha\wedge\neg\alpha))= \sneg(\nu(\alpha)\wedge\nu(\neg\alpha))$. But, by clause 4(b), the latter coincides with $\nu(\alpha\inc\neg\alpha)$, thus satisfying 4(d).} and that, moreover, any RNmatrix for $\mathcal{L}$ is structural. 

The consequence relation for $\mathcal{L}$ w.r.t. RNmatrices will be denoted by $\Vdash_{\mathcal{F}}^{\mathcal{L}}$. By the very definitions, all the axioms and the inference rule of \nbI\ are valid for the RNmatrices for $\mathcal{L}$. Moreover, the following result can be easily proved:

\begin{teo} [Soundness of $\mathcal{L}$ w.r.t. RNmatrices]
Given formulas $\Gamma\cup\{\varphi\}$ over $\Sigma_{\nbI}$, if $\Gamma\vdash_{\mathcal{L}}\varphi$, then $\Gamma\vDash_{\mathcal{F}}^{\mathcal{L}}\varphi$.
\end{teo}

We define an equivalence relation, for formulas $\Gamma\cup\{\alpha, \beta\}$ over $\Sigma_{\nbI}$, by $\alpha\equiv^{\mathcal{L}}_{\Gamma}\beta$ iff $\Gamma\vdash_{\mathcal{L}}\alpha\rightarrow\beta$ and $\Gamma\vdash_{\mathcal{L}}\beta\rightarrow\alpha$. As we have done earlier, the well-defined quotient $A^{\mathcal{L}}_{\Gamma}=\mathbf{F}(\Sigma_{\nbI}, \mathcal{V})/\equiv^{\mathcal{L}}_{\Gamma}$ becomes a Boolean algebra with the natural operations. Now, for equivalence classes of formulas $[\alpha]$ and $[\beta]$, define $\neg\alpha=\{[\neg\varphi] \ : \  \varphi\in[\alpha]\}$ and $[\alpha]\inc[\beta]=\{[\varphi\inc\psi] \ : \ \varphi\in[\alpha] \ \mbox{ and } \ \psi\in[\beta]\}$.

Evidently $\mathcal{A}^{\mathcal{L}}_{\Gamma}$ is a multialgebra for $\nbI$, as analyzed right before Theorem~\ref{complRNnbI}, so it only remains to be shown that $\mathcal{A}$ is a multialgebra for $\mathcal{L}$ in the sense of Definition~\ref{defRNmat-ext-nbI}. Consider the following:

\begin{lema}\label{equivalences in L} The following holds in  $\mathcal{A}^{\mathcal{L}}_{\Gamma}$:\\[1mm]
1. $[\alpha\inc\neg\alpha]={\sim}[\alpha\wedge\neg\alpha]$,  where $\sneg$ is the Boolean complement in $\mathcal{A}^{\mathcal{L}}_{\Gamma}$.\\
2. $[\neg(\alpha\inc\neg\alpha)]=[\alpha\wedge\neg\alpha]$ (if $\mathcal{L}=\nbIci$).\\
3. $[\neg(\alpha\wedge\neg\alpha)]=[\alpha\inc\neg\alpha]$  (if $\mathcal{L}=\nbIcl$).
\end{lema}
\begin{proof}
1. It is an easy consequence of axioms \textbf{Ip} and $\textbf{ciw}^{*}$.\\
2. By \textbf{Ip} it follows that $\alpha \land \neg\alpha, \alpha\inc\neg\alpha \vdash_{\mathcal{L}} \neg(\alpha\inc\neg\alpha)$. By definition of derivation in a Hilbert calculus, $\alpha \land \neg\alpha, \neg(\alpha\inc\neg\alpha) \vdash_{\mathcal{L}} \neg(\alpha\inc\neg\alpha)$. Using proof-by-cases and \textbf{Ax 11} it follows that $\alpha \land \neg\alpha \vdash_{\mathcal{L}} \neg(\alpha\inc\neg\alpha)$. Now, if $\mathcal{L}=\nbIci$, then it also holds that $\neg(\alpha\inc\neg\alpha)\vdash_{\mathcal{L}} \alpha \land \neg\alpha$, by $\textbf{ci}^{*}$ and \MP.\\
3. It is proved analogously to item 2, but now by using $\textbf{cl}^{*}$.
\end{proof}

\begin{cor} \label{multiL}
The multialgebra $\mathcal{A}^{\mathcal{L}}_{\Gamma}$ satisfies  conditions 1 and 2 of Definition~\ref{defRNmat-ext-nbI}.
\end{cor}
\begin{proof}
As observed above,  $\mathcal{A}^{\mathcal{L}}_{\Gamma}$ is a multialgebra for $\nbI$. Now,  if $[\beta]\in\neg[\alpha]$ then $\beta=\neg\varphi$, for a $\varphi$ such that $[\varphi]=[\alpha]$. By Lemma~\ref{equivalences in L}(1), $[\varphi\inc\neg\varphi]=\sneg[\varphi \land \neg\varphi] = \sneg([\alpha] \land [\beta])$. By definition of  $\mathcal{A}^{\mathcal{L}}_{\Gamma}$, $[\varphi\inc\neg\varphi] \in [\alpha] \inc [\beta]$, that is, $\sneg([\alpha] \land [\beta]) \in [\alpha] \inc [\beta]$. This shows that  condition~1 of Definition~\ref{defRNmat-ext-nbI}  is satisfied. Suppose now that $\mathcal{L}=\nbIci$ and $[\beta]\in\neg[\alpha]$. Then, $\beta=\neg\varphi$ for a $\varphi$ such that $[\varphi]=[\alpha]$. By definition of  $\mathcal{A}^{\mathcal{L}}_{\Gamma}$, $[\neg(\varphi\inc\neg\varphi)] \in \neg([\alpha] \inc [\beta])$. But $[\neg(\varphi\inc\neg\varphi)]=[\varphi\land\neg\varphi]$, by Lemma~\ref{equivalences in L}(2), and $[\varphi\land\neg\varphi]=[\alpha] \land [\beta]$. Hence, condition~2(a)  of Definition~\ref{defRNmat-ext-nbI} is satisfied. Finally, let $\mathcal{L}=\nbIcl$ and $[\beta]\in\neg[\alpha]$. Then, $\beta=\neg\varphi$, for a $\varphi$ such that $[\varphi]=[\alpha]$. By definition of  $\mathcal{A}^{\mathcal{L}}_{\Gamma}$ and by Lemma~\ref{equivalences in L} items~1 and~3, $\sneg([\alpha] \land [\beta]) = \sneg[\varphi \land \neg\varphi]=[\varphi \inc\neg\varphi]=[\neg(\varphi \land \neg\varphi)]$. But $[\neg(\varphi \land \neg\varphi)] \in \neg([\alpha] \land [\beta])$, hence condition~2(b)  of Definition~\ref{defRNmat-ext-nbI} is also satisfied.
\end{proof}

\begin{lema} \label{canvalL} Suppose that $\Gamma$ is $\varphi$-saturated in $\mathcal{L}$ for some formula $\varphi$.
Then, the map $\nu_{\mathcal{L}}:\mathbf{F}(\Sigma_{\nbI}, \mathcal{V}) \to  \A^{\mathcal{L}}_{\Gamma}$ such that $\nu_{\mathcal{L}}(\alpha)=[\alpha]$  is a $\Sigma_{\nbI}$-homomorphism which satisfies conditions 4(a) to 4(d) of Definition~\ref{defRNmat-ext-nbI}.
\end{lema}
\begin{proof}
It is immediate to see that $\nu_{\mathcal{L}}$ is a $\Sigma_{\nbI}$-homomorphism, since $\Gamma$ is $\varphi$-saturated in $\mathcal{L}$. Clearly $\nu_{\mathcal{L}}$ satisfies condition~4(a) of Definition~\ref{defRNmat-ext-nbI}. Conditions~4(b), 4(c) and~4(d) are also satisfied as a direct consequence of Lemma~\ref{equivalences in L}.
\end{proof}

\begin{teo}  [Completeness of $\mathcal{L}$ w.r.t. RNmatrices]
Given formulas $\Gamma\cup\{\varphi\}$ over $\Sigma_{\nbI}$, if $\Gamma\Vdash_{\mathcal{F}}^{\mathcal{L}}\varphi$, then $\Gamma\vdash_{\mathcal{L}}\varphi$.
\end{teo}
\begin{proof} It is an adaptation of the one given for Theorem~\ref{compleRNmatbI}.
Thus, suppose that $\Gamma\nvdash_{\mathcal{L}}\varphi$, and let $\Delta$ be a $\varphi$-saturated set in $\mathcal{L}$ such that $\Gamma \subseteq \Delta$. Let $\mathcal{A}^{\mathcal{L}}_{\Delta}$ be the multialgebra defined as above. By Corollary~\ref{multiL}, it generates a RNmatrix for $\mathcal{L}$ (by taking $D$ and \F\ as in Definition~\ref{defRNmat-ext-nbI}). By Lemma~\ref{canvalL}, the map $\nu_{\mathcal{L}}:\mathbf{F}(\Sigma_{\nbI}, \mathcal{V})\rightarrow A^{\bI}_{\Delta}$ such that $\nu_{\mathcal{L}}(\alpha)=[\alpha]$ is a valuation for $\mathcal{L}$ over $\mathcal{A}^{\mathcal{L}}_{\Delta}$ which is in  $\mathcal{F}$. In addition, $\nu_{\mathcal{L}}(\alpha)=1$ if and only if $\Delta\vdash_{\mathcal{L}}\alpha$. From this, $\nu_{\mathcal{L}}(\gamma)=1$ for every $\gamma \in \Gamma$ but $\nu_{\mathcal{L}}(\varphi)=0$, since $\Delta\nvdash_{\bI}\varphi$. This shows that  $\Gamma\nVdash^{\mathcal{L}}_{\mathcal{F}}\varphi$.\qedhere
\end{proof}

\subsection{Decision Methods for the Extensions of \nbI}

Let $\textbf{2}$ be the two-element Boolean algebra and consider once again the $\Sigma_{\textbf{nbI}}$-multialgebra $\textbf{2}_{\textbf{nbI}}$ defined at the beginning of Subsection~\ref{dec-nbI}. It is immediate to see that $\textbf{2}_{\textbf{nbI}}$ satisfies  conditions 1 and 2 of Definition~\ref{defRNmat-ext-nbI}.
Let $\M^\textbf{2}_{\mathcal{L}}=(\textbf{2}_{\nbI}, \{1\}, \mathcal{F}_{\textbf{2}_{\mathcal{L}}})$  be the RNmatrix for $\mathcal{L}$  defined from $\textbf{2}_{\nbI}$ according to Definition~\ref{defRNmat-ext-nbI}. The proof of the next results is immediate:

\begin{prop}
A map $\nu:\mathbf{F}(\Sigma_{\nbI},\mathcal{V})\rightarrow\{0,1\}$ is a bivaluation for $\mathcal{L}$ if, and only if, it is a valuation for the Nmatrix $(\textbf{2}_{\nbI}, \{1\})$ which lies in $\mathcal{F}_{\textbf{2}_{\mathcal{L}}}$.
\end{prop}

\begin{cor}
Given formulas $\Gamma\cup\{\varphi\}$ of $\nbI$, $\Gamma\vDash_{\mathcal{L}}\varphi$ iff $\Gamma\vDash_{\M^\textbf{2}_{\mathcal{L}}}\varphi$. Hence, $\Gamma\vdash_{\mathcal{L}}\varphi$ iff $\Gamma\vDash_{\M^\textbf{2}_{\mathcal{L}}}\varphi$.
\end{cor}

The last result  shows that the  RNmatrix $\M^\textbf{2}_{\mathcal{L}}$ induces a decision procedure for $\mathcal{L}$. Indeed, it is easy to define row-branching, row-eliminating truth tables for the logics $\nbIciw$, $\nbIci$ and $\nbIcl$, the conditions for a row to be erased being:
\begin{enumerate}
\item $\alpha\inc\beta$ and $\beta\inc\alpha$ receive different values;
\item either $\alpha\inc\neg\alpha$ or $\neg\alpha\inc\alpha$ is $0$, and either $\alpha$ or $\neg\alpha$ is $0$;
\item\begin{enumerate}
\item in the case of $\nbIci$,  $\neg(\alpha\inc\neg\alpha)$ is $1$ and either $\alpha$ or $\neg\alpha$ is $0$;
\item in the case of $\nbIcl$, $\neg(\alpha\wedge\neg\alpha)$ is $1$ and: either $\alpha\inc\neg\alpha$ is $0$, or $\neg\alpha\inc\alpha$ is $0$, or $\alpha$ and $\neg\alpha$ are both $1$.
\end{enumerate}
\end{enumerate}

The second decision procedure for $\mathcal{L}$ is by means of tableaux. Consider the following four tableau rules:
$$
\begin{array}{cp{0.5cm}cp{0.5cm}cp{0.5cm}c}
\displaystyle \frac{\textsf{0}(\varphi\inc\neg\varphi)}{\begin{array}{c}\textsf{1}(\varphi) \\ \textsf{1}(\neg\varphi)\end{array}} & & \displaystyle \frac{\textsf{0}(\neg\varphi\inc\varphi)}{\begin{array}{c}\textsf{1}(\varphi) \\ \textsf{1}(\neg\varphi)\end{array}} & & \displaystyle \frac{\textsf{1}(\neg(\varphi\inc\neg\varphi))}{\begin{array}{c}\textsf{1}(\varphi) \\ \textsf{1}(\neg\varphi)\end{array}} & & \displaystyle \frac{\textsf{1}(\neg(\varphi\land\neg\varphi))}{\textsf{0}(\varphi)\mid\textsf{0}(\neg\varphi)}  \\[2mm]
&&&&\\[2mm]
\end{array}
$$

Let $\mathbb{T}_{\nbIciw}$, $\mathbb{T}_{\nbIci}$ and $\mathbb{T}_{\nbIcl}$ be the tableau calculi obtained from $\mathbb{T}_{\nbI}$ by adding the first and second rule, the first, the second and the third rule, and the first, the second and the fourth rule, respectively. It is easy to see that $\mathbb{T}_{\mathcal{L}}$ is sound and complete for $\mathcal{L}$, being so a decision procedure for that logic.

\section{Uncharacterizability Results for Logics of Incompatibility}\label{fourth}

All the systems of logics of formal incompatibility presented here were characterized exclusively by means of (restricted) non-deterministic semantics. A natural question is  if it is possible to find a more standard semantics for these logics.  The aim of this section is showing  that these logics cannot be characterized in terms of simpler (or better behaved) semantics,  such as algebraic semantics, finite logical matrices or even finite Nmatrices.

\subsection{Logics of Incompatibility are not algebraizable}

In this section it will be shown that all the logics of incompatibility presented here are not algebraizable even in  Blok and Pigozzi's sense (\cite{BlokPigozzi}). To that end, we will use Lewin, Mikenberg and Schwarze's construction~\cite{Lewin}, and prove (as they did for $C_{1}$) that there are models of them for which the Leibniz operator $\Omega$ which  sends a filter to its largest compatible congruence is not bijective. Because of this, these logics are not algebraizable, by~\cite[Theorem~5.1]{BlokPigozzi}.

\begin{mydef}
Given a signature $\Sigma$ and a $\Sigma$-algebra $\mathcal{A}=(A, \{\sigma_{\mathcal{A}}\}_{\sigma\in\Sigma})$, a congruence in $\mathcal{A}$ is a relation $\theta$ on $A\times A$ such that, if $a_{1}\theta b_{1}$, \ldots  , $a_{n}\theta b_{n}$, then $\sigma_{\mathcal{A}}(a_{1}, \ldots  , a_{n})\theta\sigma_{\mathcal{A}}(b_{1}, \ldots  , b_{n})$.
\end{mydef}

\begin{mydef}
Given a signature $\Sigma$, a logic $\mathcal{L}$ and a $\Sigma$-algebra $\mathcal{A}=(A, \{\sigma_{\mathcal{A}}\}_{\sigma\in\Sigma})$, a $\mathcal{L}$-filter in $\mathcal{A}$ is a subset $F\subseteq A$ such that $\Gamma\vdash_{\mathcal{L}}\varphi$ implies $\Gamma\vDash_{(\mathcal{A}, F)}\varphi$, for every set of formulas $\Gamma\cup\{\varphi\}$ over $\Sigma$.
\end{mydef}

The largest compatible congruence $\theta$ with a filter $F$ is the largest congruence such that, if $a\theta b$ and $a\in F$, then $b\in F$.

So, over the signature $\Sigma_{\bI}$, we consider the $\Sigma_{\bI}$-algebra $\mathfrak{L}$ with universe $L=\{u, 1, a, b, 0\}$  where $\vee$ is the supremum and $\wedge$ is the infimum of the poset structure on $L$ given by $0<a,b<1<u$, such that $a$ and $b$ are incomparable. The other operations are given by the tables below.

\begin{figure}[H]
\centering
\begin{minipage}[t]{5cm}
\centering
\begin{tabular}{l|ccccr}
$\rightarrow$ & $u$ & $1$ & $a$ & $b$ & $0$ \\\hline
$u$ & $u$ & $u$ & $a$ & $b$ & $0$\\
$1$ & $u$ & $1$ & $a$ & $b$ & $0$\\
$a$ & $u$ & $1$ & $1$ & $b$ & $b$\\
$b$ & $u$ & $1$ & $a$ & $1$ & $a$\\
$0$ & $u$ & $1$ & $1$ & $1$ & $1$
\end{tabular}
\caption*{Table for Implication}
\end{minipage}
\begin{minipage}[t]{5cm}
\centering
\begin{tabular}{l|ccccr}
$\inc$ & $u$ & $1$ & $a$ & $b$ & $0$ \\\hline
$u$ & $0$ & $0$ & $0$ & $0$ & $1$\\
$1$ & $0$ & $0$ & $b$ & $a$ & $1$\\
$a$ & $0$ & $b$ & $b$ & $1$ & $1$\\
$b$ & $0$ & $a$ & $1$ & $a$ & $1$\\
$0$ & $1$ & $1$ & $1$ & $1$ & $1$
\end{tabular}
\caption*{Table for incompatibility}
\end{minipage}
\end{figure}

Consider the logical matrix $\mathfrak{M}=(\mathfrak{L}, D)$, with $D=\{u,1\}$. It is immediate to see that $\mathfrak{M}$ models $\bI$, meaning that $\Gamma\vdash_{\bI}\varphi$ implies that $\Gamma\vDash_{\mathfrak{M}}\varphi$.
 
\begin{lema}
$\mathfrak{L}$ has only $\nabla=L\times L$ and $\triangle=\{(x,x) \ : \  x\in L\}$ as congruences.
\end{lema}
\begin{proof}
Let $\theta$ be a congruence in $\mathfrak{L}$ such that $(x,y)$ is in $\theta$, with $x\neq y$ (that is, $\theta\neq \triangle$). It is easy to prove that $\theta=\nabla$.
\end{proof}

\begin{lema}\label{filter}
Let  $\mathcal{A}$ be a $\Sigma$-algebra and let  $\mathcal{L}$ be a logic over $\Sigma$. Then, $F$ is a $\mathcal{L}$-filter in $\mathcal{A}$ iff, for every $\Sigma$-homomorphism $\sigma:\mathbf{F}(\Sigma, \mathcal{V})\rightarrow\mathcal{A}$ the following holds:
\begin{enumerate}
\item for any instance of axiom $\psi$ of $\mathcal{L}$, $\sigma(\psi)\in F$;
\item for any instance of an  inference rule  $\psi_{1}, \ldots  , \psi_{n}/\psi$ of $\mathcal{L}$, if\\ $\sigma(\psi_{1}), \ldots  , \sigma(\psi_{n})\in F$ then $\sigma(\psi)\in F$.
\end{enumerate}
\end{lema}

Take the subsets $F_{a}=\{u, 1, a\}$ and $F_{b}=\{u, 1, b\}$ of $L$. By using Lemma~\ref{filter} we can prove  that both sets are $\bI$-filters. First of all, for any $\Sigma_{\bI}$-homomorphism $\sigma:\mathbf{F}(\Sigma_{\bI}, \mathcal{V})\rightarrow\mathfrak{L}$ and any instance of an axiom schema $\psi$ of $\bI$, since $\mathfrak{M}=(\mathfrak{L}, D)$ models $\bI$, $\sigma(\psi)\in D=\{u,1\}\subseteq F_{a}$ and $\sigma(\psi)\in D \subseteq F_{b}$, implying both $F_{a}$ and $F_{b}$ satisfy the first condition for being a $\bI$-filter. And there is only one rule of inference, \MP: one sees that, if $x\rightarrow y$ is in $F_{z}$ (for $z\in\{a,b\}$), then either $y\in F_{z}$ or $x\in L\setminus F_{z}$. So, if both $x$ and $x\rightarrow y$ are in $F_{z}$, then $y\in F_{z}$, what implies that both $F_{a}$ and $F_{b}$ are $\bI$-filters.

But $\triangle$ is the largest congruence compatible with $F_{a}$, and it is also the largest congruence compatible $F_{b}$. Indeed, $\nabla$ is not compatible with neither $F_{a}$ nor $F_{b}$, since $u\nabla 0$ and $u\in F_{a}\cap F_{b}$, but $0$ is not in $F_{a}$ nor in $F_{b}$. Clearly $\triangle$ is compatible with both $F_{a}$ and $F_{b}$, and since there are no congruences larger than $\triangle$ different from $\nabla$, we obtain the  result. Given that the Leibniz operator $\Omega_{\mathfrak{L}}$ which assigns to any filter its  largest compatible congruence is not injective, it follows by~\cite[Theorem~5.1]{BlokPigozzi} that $\bI$ is not algebraizable in the sense of Blok-Pigozzi.

To extend this result to $\nbI$ and the other logics, we consider the $\Sigma_{\nbI}$-algebra $\mathfrak{L}_{+}$ with the same universe as $\mathfrak{L}$, the same binary operations, and the negation defined as $\neg u=\neg 0=1$, $\neg 1=0$, $\neg a=b$ and $\neg b=a$. A simple calculation shows that $\mathfrak{M}_{+}=(\mathfrak{L}_{+}, D)$ models $\nbI$ and any $\mathcal{L}\in\{\nbIciw, \nbIci, \nbIcl\}$. Once again, the only congruences of $\mathfrak{L}_{+}$ are $\triangle$ and $\nabla$. It can be proven that $F_{a}$ and $F_{b}$ are $\nbI$ and $\mathcal{L}$-filters, from Lemma \ref{filter}, and $\triangle$ is the largest congruence compatible with both $F_{a}$ and $F_b$. Hence the Leibniz operator $\Omega_{\mathfrak{L}_{+}}$ is not injective, showing that neither  $\nbI$ nor $\mathcal{L}\in\{\nbIciw, \nbIci, \nbIcl\}$ are algebraizable.

\subsection{\bI\ and \nbI\ are  not characterizable by finite Nmatrices}

One might ask why we resort to RNmatrices for $\bI$, without first making use of Nmatrices. The reason is simply that no finite Nmatrix can characterize $\bI$, as it will be proved now. Suppose there is an Nmatrix $\mathcal{M}=(\mathcal{A}, D)$ that characterizes $\bI$, with $A$ the universe of $\mathcal{A}$ and $U=A\setminus D$ the set of undesignated elements. For simplicity, we will drop the indexes from the operations on $\mathcal{A}$ and use the infix notation.


\begin{lema} \label{propNmat} \ \\
1. Let $a,b \in A$. Then, either $a \to b \subseteq D$ or $a \to b \subseteq U$. Moreover, $a \to b \subseteq U$ iff $a \in D$ and $b \in U$.\\
2. For any formula $\alpha$ and valuation $\nu$, $\nu(\alpha)\in D$ iff $\nu({\sim}\alpha)\in U$.
\end{lema}
\begin{proof}
Immediate, given that the Nmatrix \M\ is a model of \cpl.
\end{proof}

\begin{lema} \label{propNmat2}
For any two elements $a, b\in A$, either $a\inc b\subseteq D$ or $a\inc b\subseteq U$;
\end{lema}
\begin{proof}
Suppose that there are values $d, u\in a\inc b$ with $d\in D$ and $u\in U$. Take then a valuation $\nu$ such that, for variables $p$ and $q$, $\nu(p)=a$, $\nu(q)=b$ and $\nu(p\inc q)=d$. But $\nu((p\inc q)\rightarrow(q\inc p)) \in d\rightarrow \nu(q\inc p)$ must be always designated, given that $\mathcal{M}$ models $\textbf{Comm}$. From Lemma~\ref{propNmat}(1), this implies that $b\inc a\subseteq D$. But then, by taking a valuation $\nu^{*}$ for which $\nu^{*}(p)=b$, $\nu^{*}(q)=a$ and $\nu^{*}(q\inc p)=u$, one reaches a contradiction.
\end{proof}

Consider now two disjoint sets of distinct variables $\{p_{n} \ : \ n\in\mathbb{N}\}$ and $\{q_{n} \ : \ n\in\mathbb{N}\}$ and the following formulas, for $i, j\in\mathbb{N}$:
\[\phi_{ij}=\begin{cases}
      p_{i}\inc q_{j} & \text{if $i<j$} \\
      {\sim}(p_{i}\inc q_{j})   & \text{otherwise}
    \end{cases}.\]
Define, for any $n\in\mathbb{N}$, $\Gamma_{n}=\{\phi_{ij} \ : \  0\leq i, j\leq n\}$. It can be proved that, for all $n$, $\Gamma_{n}\not\vdash_{\bI}p_{0}$. To see that, consider the bivaluation $\nu$ for which $\nu(p_{i})=\nu(p_{j})=0$, for all $i, j\in\mathbb{N}$, and $\nu(p_{i}\inc q_{j})$ equals $1$, if $i<j$, and $0$ otherwise. Since $\nu({\sim}\alpha)=1$ iff $\nu(\alpha)=0$, by Lemma~\ref{propNmat}(2), we find that $\nu(\phi_{ij})=1$ for all $0\leq i, j\leq n$. So $\nu(\Gamma_{n})\subseteq \{1\}$ but $\nu(p_{0})=0$. Hence $\Gamma_{n}\not\vDash_{\bI}p_{0}$ and so $\Gamma_{n}\not\vdash_{\bI}p_{0}$, by soundness.

Since \M\ characterizes \bI, the latter implies that, for any $n\in\mathbb{N}$, there exists a valuation $\nu$ for $\mathcal{A}$ such that $\nu(\Gamma_{n})\subseteq D$ but $\nu(p_{0})\in U$. However, assuming that $\mathcal{A}$ has a finite universe with $n \geq 2$ elements, there can not exist a valuation $\nu$ satisfying $\nu(\Gamma_{n})\subseteq D$. Suppose there exists such a valuation $\nu$. Given there are $n+1$ elements among $p_{0}, p_{1}, \ldots  , p_{n}$ and $\mathcal{A}$ has only $n$ elements, by the pigeonhole principle one finds there exist $0\leq i<j\leq n$ such that $\nu(p_{i})= a = \nu(p_{j})$. Since $\nu(\Gamma_{n})\subseteq D$, $\nu(\phi_{ij})=\nu(p_{i}\inc q_{j})\in D$, implying by Lemma~\ref{propNmat2} that $a\inc\nu(q_{j})\subseteq D$. From $\nu(p_{j})=a$ we get that  $\nu(p_{j}\inc q_{j})\in a\inc\nu(q_{j})$, hence  $\nu(p_{j}\inc q_{j})\in D$. Using Lemma~\ref{propNmat}(2), this implies that  $\nu(\phi_{jj})=\nu({\sim}(p_{j}\inc q_{j}))\in U$, contradicting the supposition that $\nu(\Gamma_{n})\subseteq D$. This proves the following:

\begin{teo}
There exists no finite Nmatrix which characterizes $\bI$.
\end{teo}

The latter result also implies that $\nbI$ is not characterizable by finite Nmatrices. Indeed,  if it were characterizable by some $\mathcal{M}=(\mathcal{A}, D)$, for $\mathcal{A}$ a $\Sigma_{\nbI}$-multialgebra with universe $A$, by defining $\mathcal{A}_{-}=(A, \{\sigma_{\mathcal{A}}\}_{\sigma\in\Sigma_{\bI}})$ one would find that $\mathcal{M}_{-}=(\mathcal{A}_{-}, D)$ characterizes $\bI$, contradicting our previous theorem.

\subsection{$\bI$ and \nbI\ are not characterizable by a finite Rmatrix}

RNmatrices combine aspects of both Nmatrices and Piochi's restricted matrices, or Rmatrices (see~\cite{Piochi,Piochi2}). A Rmatrix is a triple $\mathcal{M}=(\mathcal{A}, D, \mathcal{F})$ such that $\mathcal{A}$ is a $\Sigma$-algebra, $D$ is a non-empty subset of its universe and $\mathcal{F}$ is a set of homomorphisms $\nu:\mathbf{F}(\Sigma, \mathcal{V})\rightarrow\mathcal{A}$. Given formulas $\Gamma\cup\{\varphi\}$ over $\Sigma$, we say $\varphi$ is a consequence of $\Gamma$ according to $\mathcal{M}$, and write $\Gamma\vDash_{\mathcal{M}}\varphi$, if for every $\nu\in\mathcal{F}$, $\nu(\Gamma)\subseteq D$ implies that $\nu(\varphi)\in D$.

We have already proved that $\bI$ is characterizable, as every Tarskian logic is, by finite RNmatrices, but not by finite Nmatrices, so it is natural to ask whether it is characterizable by finite Rmatrices alone. The answer is negative, as we shall see. Indeed, assume that $\mathcal{M}=(\mathcal{A}, D, \mathcal{F})$ is a finite Rmatrix which characterizes $\bI$, and consider again two disjoint sets of distinct variables $\{p_{n} \ : \ n\in\mathbb{N}\}$ and $\{q_{n} \ : \ n\in\mathbb{N}\}$ and the formulas $\phi_{ij}$ and their sets $\Gamma_{n}$, for $i, j, n\in\mathbb{N}$, of the previous subsection. 

\begin{lema} \label{lemRmat}
For $\nu\in\mathcal{F}$, $\nu(\alpha)$ and $\nu({\sim}\alpha)$ can not both belong to $D$.
\end{lema}

\begin{proof}
Since ``$\sim$'' behaves classically, it satisfies $\alpha, {\sim}\alpha\vdash_{\bI}\beta$. Since $\not\vdash_{\bI}\bot_{\alpha\alpha}$, we must have that $\nu(\bot_{\alpha\alpha})\notin D$. Then, if $\nu(\alpha), \nu({\sim}\alpha)\in D$, $\nu$ would not validate the deduction $\alpha, {\sim}\alpha\vdash_{\bI}\bot_{\alpha\alpha}$, what is absurd.
\end{proof}

As we know, $\Gamma_{n}\not\vdash_{\bI}p_{0}$ for any $n\in\mathbb{N}$, which means (assuming that \M\ characterizes \bI) that there must exist a valuation $\nu\in\mathcal{F}$ for which $\nu(\Gamma_{n})\subseteq D$ but $\nu(\varphi)\not\in D$. If the universe of $\mathcal{A}$ has cardinality $n$, by the pigeonhole principle there must exist two elements $p_{i}$ and $p_{j}$ among $\{p_{n} \ : \  n\in\mathbb{N}\}$ such that $\nu(p_{i})=\nu(p_{j})$, which lead us to the following problem. Assume, without loss of generality, that $i<j$. Then, 
\[\nu(p_{i}\inc q_{j})=\nu(p_{i})\inc\nu(q_{j})=\nu(p_{j})\inc\nu(q_{j})=\nu(p_{j}\inc q_{j}).\]
Since $\nu(\Gamma_{n})\subseteq D$, $\nu(p_{i}\inc q_{j}), \nu({\sim}(p_{j}\inc q_{j}))\in D$, which means that $\nu$ is a valuation in $\mathcal{F}$ satisfying that $\nu(p_{j}\inc q_{j})$ and $\nu({\sim}(p_{j}\inc q_{j}))$ are both in $D$, a contradiction given Lemma~\ref{lemRmat}.

The conclusion is that no finite Rmatrix can characterize $\bI$. A similar argument applies to $\nbI$, as we have done at the end of the previous subsection.


\section{Translating Paraconsistent Logics}\label{fifth}

When looking at the RNmatrices introduced in the previous sections, the most important distinction was the replacement of consistency (from \lfis) for incompatibility. In this section it will be proved that there are sublogics of $\nbI$ and its extensions that capture precisely the expressiveness of $\mbC$ and its extensions. This will allow us to use the techniques for logics of incompatibility to deal with \lfis. But, more importantly, translating faithfully the latter systems into the former guarantees that the logics of incompatibility are extending, non-trivially, an important family of \lfis, showing that they deserve to be studied with greater care.

For $\mathcal{V}$ a countable set of propositional variables, we consider the (translating) function $T:\mathbf{F}(\Sigma_{\textbf{LFI}}, \mathcal{V})\rightarrow \mathbf{F}(\Sigma_{\nbI}, \mathcal{V})$ such that:

\begin{enumerate}
\item $T(p)=p$ for every $p\in \mathcal{V}$;
\item $T(\neg\alpha)=\neg T(\alpha)$;
\item $T(\alpha\#\beta)=T(\alpha)\# T(\beta)$, for any $\#\in\{\vee, \wedge, \rightarrow\}$;
\item $T(\circ\alpha)=T(\alpha)\inc\neg T(\alpha)$.
\end{enumerate}

Essentially, $T$ changes all occurrences of formulas of the form $\circ\alpha$ to $\alpha\inc\neg\alpha$.  Note that an instance  of $\textbf{bc1}$ is  translated into an instance of $\textbf{Ip}$.


\begin{prop}
$T$ is an injective function.
\end{prop}

\begin{proof}
By double induction on the complexity of $\alpha$ and $\beta$, it is proved that $T(\alpha)=T(\beta)$ implies that $\alpha=\beta$.\qedhere
\end{proof}

Is important to notice that $\alpha$ in $\mathbf{F}(\Sigma_{\nbI},\mathcal{V})$ does not belong to the set $T(\mathbf{F}(\Sigma_{\textbf{LFI}},\mathcal{V}))$ iff $\alpha$ contains a subformula $\beta_{1}\inc\beta_{2}$ such that $\beta_{2}\neq\neg\beta_{1}$. One direction is clear: if $\alpha$ contains a subformula $\beta_{1}\inc\beta_{2}$ with $\beta_{2}\neq\neg \beta_{1}$ then it is not in $T(\mathbf{F}(\Sigma_{\textbf{LFI}},\mathcal{V}))$. The converse is obtained from an induction over the complexity of $\alpha$.


\subsection{$T$ is a Conservative Translation}

A function $\mathcal{T}: \mathbf{F}(\Sigma^{1}, \mathcal{V})\rightarrow \mathbf{F}(\Sigma^{2}, \mathcal{V})$ is said to be a {\em translation} between the logics $\mathscr{L}_{1}$ over the signature $\Sigma^{1}$ and $\mathscr{L}_{2}$ over $\Sigma^{2}$, if $\Gamma\vdash_{\mathscr{L}_{1}}\varphi$ implies that $\mathcal{T}(\Gamma)\vdash_{\mathscr{L}_{2}}\mathcal{T}(\varphi)$. We say that $\mathcal{T}$ is a {\em conservative translation} if it is a translation satisfying additionally that, if $\mathcal{T}(\Gamma)\vdash_{\mathscr{L}_{2}}\mathcal{T}(\varphi)$, then $\Gamma\vdash_{\mathscr{L}_{1}}\varphi$. This definition corresponds to Definition $2.4.1$ in~\cite{ParLog} and may also be found in the original work on conservative translations, \cite{Itala}.

\begin{lema} \label{lemax}
For any endomorphism $\lambda$ of $\mathbf{F}(\Sigma_{\textbf{LFI}}, \mathcal{V})$, there exists an endomorphism $\overline{\lambda}$ of $\mathbf{F}(\Sigma_{\nbI}, \mathcal{V})$ such that $T\circ\lambda=\overline{\lambda}\circ T$.
\end{lema}
\begin{proof}
Define, for every propositional variable $p$, $\overline{\lambda}(p)=T(\lambda(p))$.
\end{proof}

\begin{teo} \label{T-translat}
Let $\mathcal{L}$ be a logic over $\Sigma_{\textbf{LFI}}$, with axiom schemata $\Psi$ and \MP\ as the only inference rule. Let $\mathcal{L}^{*}$ be the logic over $\Sigma_{\nbI}$ with axiom schemata $T(\Psi)$  and \MP\ as the only inference rule. Then, $\Gamma\vdash_{\mathcal{L}} \varphi$ implies that $T(\Gamma)\vdash_{\mathcal{L}^{*}} T(\varphi)$.
\end{teo}

\begin{proof}
Proceed by induction on the length $n$ of a proof $\alpha_{1}, \ldots  , \alpha_{n}$ of $\varphi$ in $\mathcal{L}$, showing (using Lemma~\ref{lemax} for instances of axioms) that the sequence $T(\alpha_{1}), \ldots  , T(\alpha_{n})$ is a proof of $T(\varphi)$ in $\mathcal{L}^{*}$.\qedhere
\end{proof}

The latter result proves that $T$ is always a translation from $\mathcal{L}$ to $\mathcal{L}^{*}$.

Notice that all the instances of axiom schemata of $\textbf{mbC}^{*}$ (respectively of $\textbf{mbCax}^{*}$ for an axiom $\textbf{ax}$ among $\{\textbf{ciw}, \textbf{ci}, \textbf{cl}\}$) are instances of axiom schemata of $\nbI$ (respectively $\nbI\textbf{ax}$), meaning that the systems $\nbI$ and $\nbI\textbf{ax}$ are stronger than, respectively, $\mbC$ and $\mbC\textbf{ax}$; moreover, the former are actually strictly stronger than the latter, since, e.g.,  $(\alpha\inc\alpha)\rightarrow(\alpha\rightarrow(\alpha\rightarrow\beta))$ is a tautology of the former but not the latter. More prominently, observe $\nbIcl$: it is well known that in $\mbCcl$, $\neg(\alpha\wedge \neg\alpha)$ does not imply $\neg(\neg\alpha\wedge \alpha)$ and, as we shall prove, this remains true for $\mbCcl^{*}$. But in $\nbIcl$ , from $\textbf{cl}^{*}$, $\textbf{Comm}$ and $\textbf{Ip}$ one derives
\[\vdash_{\nbIcl}\neg(\alpha\wedge\neg\alpha)\rightarrow\neg(\neg\alpha\wedge\alpha).\]


Let $\vDash_{\textbf{mbC}}$ be the consequence of \mbc\ w.r.t. its bivaluation semantics. By the results  found in Section~$2.2$ of~\cite{ParLog} we have the following:

\begin{teo} \label{complmbC}
For formulas $\Gamma\cup\{\varphi\}$ of $\textbf{mbC}$, $\Gamma\vdash_{\textbf{mbC}}\varphi$ iff $\Gamma\vDash_{\textbf{mbC}}\varphi$.
\end{teo}

This result will serve us to prove that if $T(\Gamma)\vdash_{\textbf{mbC}^{*}}T(\varphi)$, then $\Gamma\vdash_{\textbf{mbC}}\varphi$. By contraposition, suppose that $\Gamma\not\vdash_{\textbf{mbC}}\varphi$. Then, by Theorem~\ref{complmbC} there exists a bivaluation $\nu$ for $\textbf{mbC}$ such that $\nu(\Gamma)\subseteq \{1\}$ and $\nu(\varphi)=0$. We define a function $\nu^{*}$ from the formulas of $\nbI$ to $\{0,1\}$ by structural induction:

\begin{enumerate}
\item for a propositional variable $p$, $\nu^{*}(p)=\nu(p)$;
\item $\nu^{*}(\alpha\#\beta)=\nu^{*}(\alpha)\#\nu^{*}(\beta)$, for any $\#\in\{\vee, \wedge, \rightarrow\}$;
\item \begin{enumerate}
\item $\nu^{*}(\neg  T(\psi))=\nu(\neg \psi)$;
\item if $\alpha \neq T(\psi)$ for every $\psi$, $\nu^{*}(\neg \alpha)=1$;
\end{enumerate}
\item \begin{enumerate}
\item $\nu^{*}(T(\psi)\inc\neg T(\psi))=\nu^{*}(\neg T(\psi)\inc T(\psi))=\nu(\circ\psi)$;
\item in any other case, $\nu^{*}(\alpha\inc\beta)=0$.
\end{enumerate}
\end{enumerate}

\begin{prop}
If $\psi$ is a formula of $\textbf{mbC}$ then $\nu^{*}(T(\psi))=\nu(\psi)$.
\end{prop}

\begin{proof}
Proceed by induction on the complexity of $\psi$.\qedhere
\end{proof}

\begin{prop}
$\nu^{*}$, as defined above, is a bivaluation for $\nbI$.
\end{prop}

From this, $\nu^{*}(T(\Gamma))=\nu(\Gamma)\subseteq \{1\}$ but $\nu^{*}(T(\varphi))=\nu(\varphi)=0$, what implies that $T(\Gamma)\not\vDash_{\nbI}T(\varphi)$. By soundness,  $T(\Gamma)\not\vdash_{\nbI}T(\varphi)$. From this,  $T(\Gamma)\not\vdash_{\textbf{mbC}^{*}}T(\varphi)$, since $\nbI$ is strictly stronger than $\textbf{mbC}^{*}$.  This shows that, if $T(\Gamma)\vdash_{\textbf{mbC}^{*}}T(\varphi)$, then $\Gamma\vdash_{\textbf{mbC}}\varphi$.

By Theorem~\ref{T-translat} the converse also holds, and therefore $\Gamma\vdash_{\textbf{mbC}}\varphi$ iff $T(\Gamma)\vdash_{\textbf{mbC}^{*}}T(\varphi)$.

Now, by Theorem~\ref{T-translat} once again we know that $\Gamma\vdash_{\textbf{mbCax}}\varphi$ implies that $T(\Gamma)\vdash_{\textbf{mbCax}^{*}}T(\varphi)$. To prove the converse, as was done with $\textbf{mbC}$, we use bivaluations. The following result can be found in Sections~$3.1$ and~$3.3$ of \cite{ParLog} (here, $\vDash_{\textbf{mbCax}}$ denotes the consequence relation of \textbf{mbCax} w.r.t. bivaluations).

\begin{teo}
$\Gamma\vdash_{\textbf{mbCax}}\varphi$ iff $\Gamma\vDash_{\textbf{mbCax}}\varphi$.
\end{teo}

By contraposition, suppose that $\Gamma\not\vdash_{\textbf{mbCax}}\varphi$, and so there exists a bivaluation $\nu$ for $\textbf{mbCax}$ such that $\nu(\Gamma)\subseteq\{1\}$ but $\nu(\varphi)=0$. We define $\nu^{\dagger}$ from the formulas over $\Sigma_{\nbI}$ to $\{0,1\}$ by structural induction as in the case of the logic $\nbI$ and $\nu^{*}$, with the exceptions of items $(3)$ and $(4)$ which are slightly modified:

\begin{enumerate}
\item[$3^{\prime}$] $\nu^{\dagger}(\neg T(\psi))=\nu(\neg \psi)$, and if $\alpha\neq T(\psi)$ for every $\psi$:
\begin{enumerate}
\item for $\mathcal{L}=\textbf{mbCci}$ and $\alpha=\beta\inc\neg\beta$, $\nu^{\dagger}(\beta\wedge\neg \beta)=0\Rightarrow \nu^{\dagger}(\neg \alpha)=0$;
\item for $\mathcal{L}=\textbf{mbCcl}$ and $\alpha=\beta\wedge\neg \beta$, $\nu^{\dagger}(\beta\inc\neg\beta)=0\Rightarrow \nu^{\dagger}(\neg \alpha)=0$;
\item otherwise, $\nu^{\dagger}(\neg \alpha)=1$;
\end{enumerate}
\item[$4^{\prime}$]  $\nu^{\dagger}(T(\psi)\inc\neg T(\psi))=\nu^{\dagger}(\neg T(\psi)\inc T(\psi))=\nu(\circ\psi)$, and if $\alpha\neq T(\psi)$ for every $\psi$:
\begin{enumerate}
\item $\nu^{\dagger}(\alpha\inc\neg\alpha)=\nu^{\dagger}(\neg\alpha\inc\alpha)=1\Leftrightarrow \nu^{\dagger}(\alpha)=0$ or $\nu^{\dagger}(\neg \alpha)=0$; 
\item otherwise, $\nu^{\dagger}(\alpha\inc\beta)=\nu^{\dagger}(\beta\inc\alpha)=0$.
\end{enumerate}
\end{enumerate}

\begin{prop}
If $\psi\in \mathbf{F}(\Sigma_{\textbf{LFI}}, \mathcal{V})$ then $\nu^{\dagger}(T(\psi))=\nu(\psi)$.
\end{prop}

\begin{prop}
$\nu^{\dagger}$, as defined above, is a bivaluation for $\nbI\textbf{ax}$.
\end{prop}

We see that, if $\nu$ is a bivaluation for $\textbf{mbCax}$, $\nu^{\dagger}$ is a bivaluation for $\nbI\textbf{ax}$ such that $\nu^{\dagger}(T(\Gamma))=\nu(\Gamma)\subseteq\{1\}$ but $\nu^{\dagger}(T(\varphi))=\nu(\varphi)=0$. This implies that $T(\Gamma)\not\vDash_{\nbI\textbf{ax}}T(\varphi)$, hence $T(\Gamma)\not\vdash_{\nbI\textbf{ax}}T(\varphi)$. Since $\nbI\textbf{ax}$ is  stronger than $\textbf{mbCax}^{*}$, it follows that $T(\Gamma)\not\vdash_{\textbf{mbCax}^{*}}T(\varphi)$. We have therefore that, if $T(\Gamma)\vdash_{\textbf{mbCax}^{*}}T(\varphi)$ then $\Gamma\vdash_{\textbf{mbCax}}\varphi$. By Theorem~\ref{T-translat} the converse also holds, and so $\Gamma\vdash_{\textbf{mbCax}}\varphi$ iff $T(\Gamma)\vdash_{\textbf{mbCax}^{*}}T(\varphi)$, for any $\textbf{ax}\in\{\textbf{ciw}, \textbf{ci}, \textbf{cl}\}$.

The results obtained in this section show, in a precise way, that the logics of formal incompatibility presented here are more expressive than the family of \lfis\ studied, for instance, in Chapters~2 and~3 of~\cite{ParLog}, since they can faithfully encode, in  a natural way, all these systems. In this specific sense we can say that the notion of logical incompatibility is  more general than the notion of consistency (or classicality) of \lfis.

\section{Brandom's notion of incompatibility}\label{sixth}

As we mentioned above, the idea of considering the notion of ``logical incompatibility'' is not new. In particular, the philosopher Robert B. Brandom analyzed this concept in several of his works.
In this section we offer a brief overview of Brandom's work with incompatibility, echoed in Jaroslav Peregrin's research. The differences between their own methodology and ours are many. Instead of considering a binary connective for incompatibility between two sentences, they focus on incompatibility between sets of formulas. Moreover, they aim to define the notion of logical consequence from the notion of incompatibility, while we assume both to coexist. And, perhaps most importantly, although many logic systems can be retrieved from their methods, those of a paraconsistent behavior are not among them, as the only negations considered by them are classical, or at most intuitionistic. But the connection between both works is there, and quite clear: the attempt to control, in a way or another, logical explosion by mediating it through a certain notion of incompatibility.

\subsection{Brandom and Aker's ``Between saying and doing''}

In \cite{Brandom}, Brandom and Aker propose an approach to logic through incompatibility, instead of consequence; they defend that a modal understanding of incompatibility fits better with the argumentative nature of epistemology, in their neo-pragmatic program. He then proceeds to redefine consequence, and the usual connectives, starting from incompatibility as a primitive notion.

Initially, the authors start with the modal notions of commitment and entitlement: two statements are then incompatible when being committed to one implies not being entitled to the other. They then define that a statement $p$ implies $q$ whenever every statement incompatible with $q$ is incompatible with $p$, but later further generalize that: they believe that incompatibility must deal with pairs of sets of statements, their argument being that a claim may be incompatible with a set of claims without being incompatible with any particular element. This is not of concern to us, since we do not aim to encompass every reasoning involving incompatibility, but rather prefer to focus on its interplay with paraconsistency and its possible semantics. \cite{Brandom} demands only two properties of incompatibility, derived from its natural interpretation and intuition:
\begin{enumerate}
\item if $X$ is incompatible with $Y$, $Y$ is incompatible with $X$ (symmetry);
\item if $X$ is incompatible with $Y$, and $Z$ contains $Y$, then $X$ is incompatible with $Z$ (persistence).
\end{enumerate}

In our logics of incompatibility, symmetry is analogous to the commutative axiom $\textbf{Comm}$; regarding persistence, we may, for sets of formulas $X$ and $Y$, write a generalized incompatibility operator $X\inc Y$ whenever there exist formulas $\alpha$ and $\beta$ such that $X\vdash \alpha$, $Y\vdash\beta$ and $\alpha\inc\beta$. Then, persistence is reobtained on this environment. Of course, this is not to say that our incompatibility logics naturally model Brandom and Aker's approach, but rather to show its versatility and expressibility.

The authors stress that their incompatibility should not be limited to truth values, but they still mention the interpretation of $p$ and $q$ being incompatible as the impossibility of $p$ and $q$'s conjunction. This is very distant from our own take on incompatibility, as it limits incompatibility to the modal formula $\square{\sim}(p\wedge q)$. To further distance its incompatibility from ours, \cite{Brandom} takes as a tacit starting point that a statement should be incompatible with its negation, what may be justified if they intend to replicate merely classical negation, but disregards non-classical (paraconsistent) negations and assumes prior connections between incompatibility and negation. From there, the negation of a claim $p$ is defined as the minimum (with respect to deduction) claim incompatible to $p$. Notice that, in $\bI$, if $p$ and $q$ are incompatible, then $p$ implies the negation of $q$: this may seem awfully close to Brandom and Aker's account. However, also notice that ${\sim}q$ is not, necessarily, incompatible with $q$, and we do not intend to define negation from this relationship. It becomes clear that~\cite{Brandom} has no interest in dealing with paraconsistency when it defines an inconsistent set of formulas as any set which derives both a claim and its negation.


The authors define the conjunction of statements $p$ and $q$ as the minimal statement incompatible with every set $X$ incompatible with $\{p,q\}$: the logic obtained from these connectives is proved to be $\textbf{CPL}$, what shows one may recover classical accounts of logic from the notion of incompatibility. And yet, they points out that, in their logic of incompatibility, connectives do not have the semantic sub-formula property, meaning they are not truth-functional. This non-deterministic behavior is frequent among \lfis\ and very clearly present in our logics of incompatibility.

\subsection{Peregrin's ``Brandom's incompatibility semantics''}

In this first article \cite{Peregrin1},  Peregrin studies Brandom and Aker's notion of incompatibility from a more philosophical standpoint. More precisely, his main concerns are tied to the problem of whether formal semantics are truly compatible with pragmatic and inferentialist views. His argument may be summarized as stating that, as merely a model for natural processes, formal semantics can indeed be used by the pragmatist, as long as the distinction between model and what is modeled is not ignored.

Peregrin stresses how the usual approach to incompatibility in logic is to say that sets of formulas $X$ and $Y$ are incompatible whenever their union can deduce anything. Much in line with our own views of the problem, Peregrin defends that reducing incompatibility to inference is, first of all, wasteful, as it disregards the possible intricacies the concept may carry. Second, incompatibility, as a byproduct of inference, becomes dependent on how expressive is the logic we work over.
 
To us, one of the most important developments found in this article is the connection established between incompatibility and Kripke semantics. Peregrin defines a possible world, once a concept of coherence is given, as a maximal coherent set of formulas. The truth of a statement on a given world is then taken to be the belonging of this statement to the world. Reciprocally, he derives a notion of incoherence from a semantics of possible worlds by saying that a set of formulas is incoherent if no formula of this set is validated in any world. Peregrin also points out how the necessity of a statement $p$ in \cite{Brandom} is equivalent, in his semantic of possible words,  to validity. This of course means that Brandom and Aker's take on modal logic trough incompatibility only derives the most basic modalities. He offers an interesting alternative, in order to add richness to those modal logics: a second level of incompatibility, or rather incoherence, a meta-coherence if you will. With this, he is able to characterize modal logics more complex than $S5$.

\subsection{Peregrin's ``Logic as based on incompatibility''}

In this second article by Peregrin of great interest to us, the author defends that the most natural logic to emerge from defining incompatibility through inference is intuitionistic, and that ``Between saying and doing'' reaches a logic of, instead of intuitionistic, classical character only due to Brandom and Aker's method, and not to the nature of incompatibility.

Peregrin~\cite{Peregrin2} defines then an environment that should deal, simultaneously, with inference and incompatibility: a triple $(S, \bot, \vdash)$ where $S$ is a set, $\bot$ is a set of inconsistent subsets of $S$ and $\vdash$ is a deduction relation on $S$. The conditions he requires of these triples are: persistence of incompatibility, and the cut rule and extensiveness of the deduction relation. Peregrin suggests a possible interplay between the two concepts:
\begin{enumerate}
\item if $X$ is inconsistent, it trivializes any deduction;
\item if $X$ deduces $p$, $Y\cup\{p\}$'s inconsistency implies $X\cup Y$'s inconsistency.
\end{enumerate}
The first condition is very in line with what we expect of incompatibility: that it trivializes an argument. The second condition corresponds to defeasibility in \cite{Brandom}. The author sustains that  these conditions reduce incompatibility to inference and vice-versa. Of negation, Peregrin demands: $(1)$ that $\{p, {\sim}p\}$ is inconsistent; and $(2)$ that if $X\cup\{p\}$ is inconsistent, $X$ deduces ${\sim}p$. With this, we see that the flexibility that the author is hoping to obtain by modifying Brandom and Aker's stipulations does not encompass paraconsistency, nor is this his objective, as far as we can see. He seems, instead, more concerned with modal systems, allowing our logics to fill a gap in his approach.

The additional requirement that $X\cup\{{\sim}p\}$ being inconsistent implies $X\vdash p$ is then equivalent to stating that the negation in question is of classical behavior; Peregrin uses this to show that ``Between saying and doing'' could only obtain classical negation. The author suggests that varying the techniques found in his article could lead to relevant and even linear logic, but not paraconsistent ones.

\subsection*{Conclusion and Future Work}

On the realm of logics of incompatibility, we have here provided only the most basic systems we could think of applying such a notion, and many more seem to be plausible. As is done with the consistency operator when dealing with logics of formal inconsistency, axioms concerning the propagation of incompatibility are a fruitful area of study. One more noticeable thing is that we have provided, as semantics for logics with the generalized Sheffer's stroke \inc, both bivaluations and RNmatrices, but we have not yet studied applying to those systems other kind of semantics such as Kripke semantics or finite sets of finite Nmatrices.

In a different direction, it seems that a deeper philosophical analysis of incompatibility is needed, maybe along the lines of~\cite{O'Connor}: what are its epistemological and historical backgrounds? Can it be described through modalities, or does it have a natural relationship with them? How entrenched is this notion in science and, in particular, in mathematics? Is it an intrinsic notion, and if so, is it better described as a primitive or derived notion? Those, and many other topics of study, seem to ramify from the research presented in this article.

\paragraph{Acknowledgements.} 

The first author acknowledges support from  the  National Council for Scientific and Technological Development (CNPq), Brazil
under research grant 306530/2019-8. The second author was supported by a doctoral scholarship from CAPES, Brazil.

\end{document}